\documentclass[11pt,oneside,reqno,a4paper]{amsart}
\usepackage[T1]{fontenc}
\usepackage[utf8]{inputenc}
\usepackage[english]{babel}
\usepackage{setspace}
\usepackage[margin=1in,headsep=0.4in,footskip=0.2in]{geometry}
\usepackage{graphicx}
\usepackage{bm}
\usepackage[dvipsnames]{xcolor}
\usepackage{amsmath,amssymb,amsthm}
\usepackage{enumitem}
\usepackage{amsfonts}
\usepackage{bbm}

\usepackage{siunitx}
\usepackage{hyperref}
\hypersetup{
 colorlinks,
 citecolor=RoyalBlue,
 linkcolor=RoyalBlue,
 urlcolor=RoyalBlue}
\usepackage{esint}
\usepackage{lmodern}
\usepackage{bbm}

\usepackage[
  rm={oldstyle=false,proportional=true},
  sf={oldstyle=true,proportional=true},
  tt={oldstyle=true,proportional=true,variable=true},
  qt=false,
]{cfr-lm}
\usepackage{cite,enumerate,setspace}
\newcommand{\tmabbr}[1]{#1}
\newcommand{\assign}{:=}

\newcommand{\mathd}{\mathrm{d}}

\newcommand{\of}{:}

\newcommand{\tmem}[1]{{\em #1}}
\newcommand{\tmmathbf}[1]{\ensuremath{\boldsymbol{#1}}}

\newcommand{\tmop}[1]{\ensuremath{\operatorname{#1}}}

\theoremstyle{definition}
\newtheorem{definition}{Definition}[section]
\newtheorem{remark}{Remark}[section]
\newcommand{\dual}[2]{\langle \mathrm{#1},\mathrm{#2}\rangle}

\newcommand{\duall}[2]{\langle {#1},{#2}\rangle}

\newcommand{\curl}{\ensuremath{\mathbf{{curl\,}}}}
\newcommand{\curle}{\ensuremath{\mathbf{curl}_\varepsilon\,}}

\newcommand{\desf}[1]{\emph{#1}}
\newcommand{\chcM}{\tmmathbf{\Phi}_{\mathcal{M},{\varepsilon}}}
\newcommand{\chcMmone}{{\tmmathbf{\Phi}_{\mathcal{M},\varepsilon}^{- 1}}}
\newcommand{\Keps}{\tmmathbf{h}_{\varepsilon}}
\newcommand{\held}{{\mathfrak{d}}}
\newcommand{\heffe}{{\tmmathbf{h}_{\mathsf{eff}}^{\varepsilon}}}
\newcommand{\heffzero}{{\tmmathbf{h}_{\mathsf{eff}}^0}}
\newcommand{\heff}{{\tmmathbf{h}_{\mathsf{eff}}}}
\newcommand{\hd}{{\tmmathbf{h}_{\mathsf{d}}}}
\newcommand{\chc}{{\tmmathbf{\Phi}}}
\newcommand{\ds}{{\partial_s}}
\newcommand{\m}{\tmmathbf{m}}
\newcommand{\f}{{\tmmathbf{f}}}
\newcommand{\tmu}{\tmmathbf{u}}

\newcommand{\ue}{\tmmathbf{u}_{\varepsilon}}
\newcommand{\ues}{\tmmathbf{u}_{\varepsilon}^{\ast}}

\newcommand{\grad}{{\nabla}}
\newcommand{\Fsp}{\tmmathbf{{\of}}}
\newcommand{\M}{\tmmathbf{M}}
\newcommand{\Stwo}{\mathbb{S}}
\newcommand{\RR}{\mathbbmss{R}}
\newcommand{\NN}{\mathbbmss{N}}
\newcommand{\eqs}{=}

\newcommand{\ppsi}{\boldsymbol{\psi}}
\newcommand{\pphi}{\boldsymbol{\phi}}
\renewcommand{\v}{\boldsymbol{v}}
\newtheorem{algorithm}{Algorithm}

\numberwithin{equation}{section}
\theoremstyle{plain}
\newtheorem{theorem}{Theorem}[section]
\newtheorem{proposition}[theorem]{Proposition}

\makeatletter
\def\@seccntformat#1{\hspace*{0mm}%
 \protect\textup{\protect\@secnumfont
   \ifnum\pdfstrcmp{subsection}{#1}=0 \bfseries\fi
   \csname the#1\endcsname
   \protect\@secnumpunct
     }%
}

\makeatletter
\g@addto@macro \normalsize {%
 \setlength\abovedisplayskip{10pt plus 2pt minus 2pt}%
 \setlength\belowdisplayskip{10pt plus 2pt minus 2pt}%
}
\makeatother

\begin{document}

\title{Micromagnetics of thin films in the presence of Dzyaloshinskii--Moriya interaction}
\author{Elisa~Davoli, Giovanni~Di~Fratta, Dirk~Praetorius, Michele~Ruggeri}
\address{TU Wien, Institute of Analysis and Scientific Computing,
Wiedner Hauptstra{\ss}e 8--10, 1040, Vienna, Austria}
\email{elisa.davoli@asc.tuwien.ac.at}
\email{giovanni.difratta@asc.tuwien.ac.at}
\email{dirk.praetorius@asc.tuwien.ac.at}
\email{michele.ruggeri@asc.tuwien.ac.at}
\keywords{Micromagnetics, Landau--Lifshitz--Gilbert Equation, Magnetic thin films, Dzyaloshinskii--Moriya interaction, $\Gamma$-Convergence, Finite elements}
\subjclass[2010]{49S05, 35C20, 35Q51, 82D40, 65M12, 65M60}

\begin{abstract}
In this paper, we study the thin-film limit of the micromagnetic energy functional in the presence of bulk Dzyaloshinskii--Moriya interaction (DMI). Our analysis includes both a stationary $\Gamma$-convergence result for the micromagnetic energy, as well as the identification of the asymptotic behavior of the
 associated Landau--Lifshitz--Gilbert equation.
In particular, we prove that, in the limiting model, part of the DMI term behaves like the projection of the magnetic moment onto the normal to the film, contributing this way to an increase in the shape anisotropy arising from the magnetostatic self-energy.
Finally, we discuss a convergent finite element approach for the approximation of the time-dependent case and use it to numerically compare the original three-dimensional model with the two-dimensional thin-film limit.
\end{abstract}

\maketitle

\section{Introduction}

\subsection{Chiral effects in micromagnetics}

Due to an increasing interest in spintronics applications, magnetic skyrmions are currently the subject of intense research activity, spanning from mathematics to physics and materials science. Although these chiral structures can emerge in different experimental settings, our mathematical analysis is centered on thin films derived from bulk materials without inversion symmetry, where the Dzyaloshinskii--Moriya interaction (DMI) can twist the otherwise ferromagnetic spin arrangement (cf.~\cite{Back_2020}).
 
In this paper, we identify a limiting model for micromagnetic thin films in the presence of bulk DMI.
Our investigations focus on the limiting behavior of the observable magnetization states
as the thickness parameter tends to zero.
The analysis includes both a stationary $\Gamma$-convergence result for the micromagnetic energy,
as well as the identification of the asymptotic behavior of the associated Landau--Lifshitz--Gilbert (LLG) equation.
In particular, the derived limiting model unveils some physics.
Indeed, quite unexpectedly, we find that part of the DMI behaves like the projection of the magnetic moment onto the normal to the film, contributing this way to an increase in the shape anisotropy originating from the magnetostatic self-energy. Our analytical derivation is complemented by a numerical approximation of the limiting solution via a projection-free tangent plane scheme, and the results of numerical simulations are discussed.

Before stating our main results, we set up the physical framework and introduce some notation.
In the variational theory of micromagnetism ({\tmabbr{cf.}} {\cite{BrownB1963,hubert2008magnetic,LandauA1935}}),
the observable states of a rigid ferromagnetic body occupying a region
$\Omega \subseteq \RR^3$ are described by its magnetization $\M$, a vector
field verifying the so-called \tmem{saturation constraint}:
There exists a material-dependent positive constant $M_s$
such that $| \M | = M_s$ in $\Omega$.
The {\tmem{saturation magnetization}} $M_s \assign M_s (T)$  depends
on the temperature $T$, but vanishes above the so-called {\tmem{Curie temperature}} $T_c$
which is characteristic of each crystal type. However, when the ferromagnet is
at a fixed temperature well below $T_c$, the
value of the saturation magnetization can be considered constant in $\Omega$. Therefore, we can
express the magnetization in the form $\M \assign M_s \m$, where $\m : \Omega
\rightarrow \Stwo^2$ is a vector field taking values in the unit sphere
$\Stwo^2$ of $\RR^3$. 

Although the modulus of $\m$ is constant in space, in general, it is not the
case for its direction. In single-crystal ferromagnets
({\tmabbr{cf.}}~{\cite{AcerbiA2006,alouges2015homogenization,Davoli_2020}}), the
observable magnetization states can be described as the
local minimizers of the micromagnetic energy functional which, after a
suitable nondimensionalization, reads as 
\begin{equation}
  \mathcal{G}_{\tmop{sym},\Omega} ( \m ) \assign \underset{= :
  \mathcal{E}_{\Omega} ( \m )}{\frac{1}{2} \int_{\Omega} |
  \nabla \m |^2 \mathd x } + \underset{= : \mathcal{A}_{\Omega} ( \m
  )}{\int_{\Omega} \varphi_{\tmop{an}} ( \m )\, \mathd x} 
  \underset{= : \mathcal{W}_{\Omega} ( \m )}{- \frac{1}{2}
  \int_{\Omega} \hd [ \m \chi_{\Omega} ] \cdot \m\, \mathd x} 
  \underset{= : \mathcal{Z}_{\Omega} ( \m )}{- \int_{\Omega}
  \tmmathbf{h}_a \cdot \m\, \mathd x}  \label{eq:GLunorm}
\end{equation}
with $\m \in H^1 ( \Omega, \Stwo^2 )$, and {where} $\m \chi_{\Omega}$ denotes the
extension of $\m$ by zero to the whole space outside $\Omega$. 

The {\tmem{exchange energy}} $\mathcal{E}_{\Omega} ( \m )$
penalizes nonuniformities in the orientation of
the magnetization. The {\tmem{magnetocrystalline anisotropy energy}}
$\mathcal{A}_{\Omega} ( \m )$ accounts for the existence of preferred
directions of the magnetization. In general, $\varphi_{\tmop{an}} : \Stwo^2 \rightarrow \RR_+$
is assumed to be a nonnegative Lipschitz continuous function that vanishes only on a distinguished set of directions known as easy axes.
The quantity $\mathcal{W}_{\Omega} ( \m )$ represents the
{\tmem{magnetostatic self-energy}} and describes the energy due to the
stray field $\hd [ \m \chi_{\Omega} ]$
generated by $\m \chi_{\Omega} \in L^2 ( \RR^3, \RR^3 )$. The stray field can be characterized as the unique solution in $L^2 ( \RR^3, \RR^3)$ of the Maxwell--Amp\`ere equations of magnetostatics \cite{BrownB1962,praetorius2004analysis,Di_Fratta_2019}:
\begin{equation}
  \left\{ \begin{array}{l}
    \tmop{div} \tmmathbf{b} [ \m \chi_{\Omega} ] \eqs 0,\\[2pt]
    \curl  \hd [ \m \chi_{\Omega}] \eqs \tmmathbf{0},\\[2pt]
    \tmmathbf{b} [ \m \chi_{\Omega}] = \mu_0 ( \hd [ \m \chi_{\Omega}] + \m\chi_{\Omega}
    ),
  \end{array} \right.  \quad \label{eq:FarMaxdemag}
\end{equation}
where $\tmmathbf{b} [ \m \chi_{\Omega}]$ denotes the magnetic flux density, and
$\mu_0$ is the vacuum permeability. The
linear operator $\hd : \m\chi_{\Omega} \mapsto \hd [ \m \chi_{\Omega}]$ is then a bounded, nonlocal, and
{\tmem{negative-definite}} operator which satisfies the following energy bounds:
\begin{equation}
  0 \leqslant \int_{\RR^3} | \hd [ \m \chi_{\Omega}] |^2\,\mathd x \eqs -
  \int_{\Omega} \hd [ \m \chi_{\Omega}] \cdot \m \,\mathd x\leqslant \| \m
  \|^2_{L^2 (\Omega)} . \label{eq:hdrels}
\end{equation}
Finally, the term $\mathcal{Z}_{\Omega} ( \m )$ is the
{\tmem{Zeeman energy}} and models the tendency of the specimen to have
its magnetization aligned with the (externally) applied field
$\tmmathbf{h}_a$.
Overall, the variational analysis
of {\eqref{eq:GLunorm}} arises as a nonconvex and nonlocal problem.

In this work, other than the classical energy terms in
$\mathcal{G}_{\tmop{sym},\Omega}$,  we consider the possible lack of centrosymmetry
in the crystal lattice structure of the ferromagnet. Generally
speaking, this can be done by superimposing to the energy density in
{\eqref{eq:GLunorm}} the contribution from suitable Lifshitz
invariants of the chirality tensor $\nabla \m \times \m$. Here, we are
interested in the {\desf{bulk $\tmop{DMI}$}} {contributions,} whose energy density is
given by the trace of {the} chirality tensor. Precisely, for every $\m \in H^1
( \Omega, \Stwo^2 )$, we define the {functional}
\begin{equation}
  \mathcal{D}_{\Omega} ( \m ) \assign \kappa \int_{\Omega} \curl 
  \m \cdot \m\, \mathd x.
\end{equation}
The normalized constant $\kappa \in \RR$ is the so-called {\desf{DMI
constant}}, and its sign determines the chirality of the system. The full
micromagnetic energy functional we are interested in is then
$ \mathcal{G}_{\tmop{sym},\Omega}
( \m ) +\mathcal{D}_{\Omega} ( \m )$.
Since one of our main aims concerns the derivation via
$\Gamma$-convergence of a 2D model, to streamline the presentation of the results,
we will neglect the energy terms $\mathcal{A}_{\Omega}
( \m )$ and $\mathcal{Z}_{\Omega} (\m)$ because in
this dimension reduction context they act as $\Gamma$-continuous perturbations \cite{dal1993introduction} and their $\Gamma$-continuous limit is straightforward to compute.
Summarizing, for {every} $\m \in H^1 ( \Omega, \Stwo^2 )$, we consider
the micromagnetic energy functional 
\begin{equation}
  \mathcal{G}_{\Omega} ( \m ) \assign \frac{1}{2} \int_{\Omega}
  | \nabla \m |^2\, \mathd x+ \kappa \int_{\Omega} \curl  \m \cdot \m\, \mathd x {-}
  \frac{1}{2} \int_{\Omega} \hd [ \m \chi_{\Omega} ] \cdot \m\,\mathd x  .
 \label{eq:GLunormfull}
\end{equation}
The existence of at least one minimizer of $\mathcal{G}_{\Omega} ( \m
)$ in $H^1 ( \Omega, \Stwo^2 )$ easily follows by the
direct method of the calculus of variations. Indeed, although the bulk DMI  energy  $\mathcal{D}_{\Omega}$ is,
\textsl{a~priori}, neither positive nor negative,  it is linear in the partial derivatives and can be controlled by the exchange energy.
Namely, by the identity $\curl  \m= \sum_{i = 1}^3 e_i\times \partial_i\m$, one obtains
\begin{align}
\mathcal{E}_{\Omega} ( \m ){+}\mathcal{D}_{\Omega} ( \m )&\eqs 
  \frac{1}{2} \sum_{i = 1}^3\int_{\Omega}  | \held_i \m |^2\,\textrm{d}x - \frac{1}{2}
  \kappa^2 \sum_{i = 1}^3 \int_{\Omega}| e_i \times \m |^2\,\textrm{d}x 
 \nonumber\\
  & \eqs  \frac{1}{2} \sum_{i = 1}^3  \int_{\Omega}| \held_i \m
  |^2\,\textrm{d}x - \kappa^2|\Omega|, \label{eq:helicalder}
\end{align}
for every $\m \in
H^1 ( \Omega, \Stwo^2 )$, where, for $i = 1, 2, 3$, the quantities 
\begin{equation} \label{eq:helders}
\held_i \m \assign \partial_i \m - \kappa ( e_i
\times \m )
\end{equation}
represent the so-called {\desf{helical derivatives}}~\cite{Melcher14,Di_Fratta_2019b}. The name is motivated by the fact that $\mathfrak{d}_i \m = 0$ for helical fields $\m$ which perform a rotation of constant frequency $\kappa$ perpendicular to the $e_i$ axis, that is, counter-clockwise or clockwise if the sign of $\kappa$ is positive or
negative, respectively. Note that they can also be regarded as a variant
of covariant derivatives emerging in the context of gauged sigma models~\cite{schroers1995bogomol}.

\subsection{State of the art} 

The analysis of micromagnetic thin films is a subject with a long history.
It dates back to the seminal papers \cite{GioiaJames97, carbou2001thin},
where the authors show that in planar thin films,
the effect of the demagnetizing field operator drops down to an easy-surface anisotropy term.
Landau states in thin ferromagnetic films are investigated in \cite{IgnatOtto11}.
A complete reduced theory for thin films micromagnetics has been established in \cite{DeSimoneKohnMuellerOtto02}.
Recently, in \cite{HadijiShirakawa10} the authors have considered static dimension
reduction under possible degeneracy of the material coefficients.
A thorough analysis of soft ferromagnetic films has been the subject
of~\cite{Melcher-reg07, KohnSlastikov05, DeSimoneETal01}
(see also \cite{Slastikov05} for the case of nonuniformly extruded thin shells).
The associated thin-film dynamics has been analyzed
in~\cite{KohnSlastikov-dyn05,Melcher10,EGarcaCervera2001EffectiveDF}.
In \cite{CapellaMelcherOtto07,CoteIgnatMiot14}, the authors characterize Neel walls in ferromagnetic thin films.
The various regimes arising as effect of the presence of external fields
are the focus of~\cite{CanteroOtto06,CanteroOttoSteiner07}.
The study of domain walls has been undertaken
in~\cite{KurzkeMelcherMoser06,Ignat09,IgnatOtto07,KnuepferMuratovNolte19}
(see also \cite{LundMuratov16} for the case of 1D walls
under fourfold anisotropy and \cite{LundMuratovSlastikov20} for ultrathin films),
whereas that of boundary vortices is carried out in \cite{Kurzke06,KurzkeMelcherMoserSpirn11,Moser04,Moser05}.
The effects of periodic surface roughness have been studied in \cite{MoriniSlastikov18}.
In \cite{Muratov19}, the author examines second-order phase transitions with dipolar interactions.
We also refer to \cite[Section 1.2]{KruzikProhl06} and the references therein
for a review on the mathematics of magnetic \emph{planar} thin films.
Magnetic \emph{curved} thin films have been the object of extensive investigations in recent times,
because of their capability to induce an effective antisymmetric interaction even in the absence of DMI
(see, e.g., \cite{Di_Fratta_2019, Di_Fratta_2020, Melcher_2019} and the topical review \cite{Streubel_2016}).
We finally mention \cite{dkps} and \cite{bresciani} for the magnetoelastic case.

Due to the nonlinear, nonconvex, and nonlocal nature of the problem,
the computation of minimizers of the micromagnetic energy
and the numerical approximation of the LLG equation
are challenging tasks.
In the last twenty years,
these problems have been the subject of several mathematical studies;
see, e.g., the monograph~\cite{prohl2001},
the review articles~\cite{KruzikProhl06,carlos2007},
the papers~\cite{bp2006,alouges2008,multiscale2014,akst2014,gao2014,gaussseidel2020},
and the references therein.
As far as the numerical analysis of the LLG equation in the presence of DMI is concerned,
we refer to the recent work~\cite{hrkac2017convergent}. 

\subsection{Contributions of the present work}
For $\varepsilon > 0$, consider the thin-film domain $\Omega_{\varepsilon} \assign \omega
\times (0, \varepsilon)$, where $\omega$ is a bounded domain in $\RR^2$ having
Lipschitz boundary, and assume that $\Omega_{\varepsilon}$ is a thin specimen
of a ferromagnetic body. As it is customary in dimension reduction, we
introduce the family of diffeomorphisms $\chc_{\varepsilon} : (\sigma, s) \in
\RR^2 \times \RR \rightarrow (\sigma, \varepsilon s_{}) \in \RR^2 \times \RR$.
We denote by $\chcM$ the restriction of $\chc_{\varepsilon}$ to the set
$\mathcal{M} \assign \omega \times I$, $I = (0, 1)$.
To every
$\m_{\varepsilon} \in H^1 ( \Omega_{\varepsilon}, \Stwo^2 )$, we
associate a map $\tmu \in H^1 ( \mathcal{M}, \Stwo^2 )$ given by
$\tmu \assign \m_{\varepsilon} \circ \chcM $.
Hereinafter, we  use the convention  that when we talk about weak convergence in $H^1 (\mathcal{M}, \Stwo^2 )$
we refer to the topology induced in $H^1 (\mathcal{M}, \Stwo^2 )$ by the weak topology
of $H^1 ( \mathcal{M}, \RR^3 )$;
a similar remark is understood for strong convergence in $H^1 (\mathcal{M}, \Stwo^2 )$.
Note that, owing to the Rellich theorem, $H^1 (\mathcal{M}, \Stwo^2 )$ is a weakly closed subset
of $H^1 ( \mathcal{M}, \RR^3 )$. 

The first result of this paper concerns the variational characterization of the asymptotic behavior of a rescaling of the sequence $\mathcal{G}_{\Omega_{\varepsilon}}$ in the
limit for $\varepsilon \rightarrow 0$. This amounts to the identification of the $\Gamma$-limit of the family of energy functionals defined by
\begin{equation}
  \mathcal{F}_{\varepsilon} ( \tmu ) := \frac{1}{2}
  \int_{\mathcal{M}} | \nabla_{\varepsilon} \tmu |^2\, \mathd x+ \kappa
  \int_{\mathcal{M}} \curl_{\varepsilon}  \tmu \cdot \tmu\,\mathd x +  \frac{1}{2}
  \int_{\RR^3} | \Keps [ \tmu \chi_{\mathcal{M}}] |^2\,\mathd x,
  \label{eq:Functional}
\end{equation}
 for every $\tmu \in H^1 (
\mathcal{M}, \Stwo^2 )$, where $\Keps [ \tmu \chi_{\mathcal{M}} ] \assign \hd [
\m_{\varepsilon} \chi_{\Omega_{\varepsilon}} ] \circ \chc_{\varepsilon}
\eqs \hd [ ( \tmu \chi_{\mathcal{M}} ) \circ \chcMmone ]
\circ \chc_{\varepsilon}$, $\nabla_{\varepsilon}$ is the
$\varepsilon$-rescaled gradient operator given by $\nabla_{\varepsilon} =
( \partial_1, \partial_2, \varepsilon^{- 1} \ds
)^{\mathsf{T}}$ and $\curle  \tmu = \nabla_{\varepsilon}
\times \tmmathbf{u}$. Note that $\mathcal{F}_{\varepsilon} ( \tmu
) = \frac{1}{\varepsilon} \mathcal{G}_{\Omega_{\varepsilon}} (
\m_{\varepsilon} )$. We show that the asymptotic behavior of minimizers of $\mathcal{F}_{\varepsilon}$ is encoded in the minimizers of an effective functional with a complete local character.
Precisely, our first result is stated in the following theorem.

\begin{theorem}
\label{thm:main1}
The family of functionals $(\mathcal{F}_{\varepsilon})_{\varepsilon>0}$ is equicoercive in the weak topology of $H^1 (\mathcal{M}, \Stwo^2 )$ and  $\Gamma$-converges, weakly in $H^1 (\mathcal{M}, \Stwo^2 )$, to the energy functional $\mathcal{F}_0$ defined by
\begin{equation}
  \mathcal{F}_0 ( \tmu ):= \frac{1}{2} \int_{\omega} |
  \nabla_{\omega} \tmu (\sigma) |^2 \mathd \sigma + \kappa
  \int_{\omega} \curl_{\omega}  \tmu \cdot \tmu\, \mathd\sigma+ \left(\frac{1 + \kappa^2}{2}\right)
  \int_{\omega} ( \tmu (\sigma) \cdot e_3 )^2 \mathd \sigma-\frac{\kappa^2}{2}|\omega|,
  \label{eq:gammalimit}
\end{equation}
if $\tmu\in H^1 ( \mathcal{M}, \Stwo^2 )$ is independent of the $s$-variable, and $\mathcal{F}_0 (\tmu) = + \infty$ otherwise. In the expression above, 
$\nabla_\omega:=(\partial_1,\partial_2)$ and $ \curl_\omega:=e_1 \times \partial_1+e_2 \times \partial_2$,
denote the corresponding differential operators meant with respect to the tangential variables $(x_1,x_2)\in \omega$. 
In particular, if for every $\varepsilon$ the map  $\ue \in H^1(\mathcal{M}, \Stwo^2)$  is a  minimizer of $\mathcal{F}_{\varepsilon}$, then, upon possible extraction of a subsequence, $(\ue)_{\varepsilon>0}$ converges strongly in $H^1(\mathcal{M}, \Stwo^2)$ to a minimizer of $\mathcal{F}_0$.
\end{theorem}

\begin{remark} \label{rem:curl2D}
Note that the definition of $\curl_\omega = e_1 \times \partial_1 + e_2 \times \partial_2$ in Theorem~\ref{thm:main1} is formally consistent with the 3D curl if $\partial_3\tmu=0$.
\end{remark}

\begin{remark}
We stress that part of the DMI energy in \eqref{eq:Functional} contributes, in the limiting energy \eqref{eq:gammalimit}, to an increase in the shape anisotropy of the thin film through the energy density $\kappa^2( \tmu (\sigma) \cdot e_3 )^2$. Also, we observe that in contrast to the classical setting $\kappa =0$ (cf.~\cite{GioiaJames97, carbou2001thin}), where constant in-plane magnetizations are the minimizers of $\mathcal{F}_0$
(and higher-order terms in the asymptotic expansions are needed to gather more information on the direction of the magnetization), here the presence of the DMI makes the reduced energy functional $\mathcal{F}_0$ nontrivial.
\end{remark}

 The proof strategy essentially relies on the notion of $\Gamma$-convergence (see \cite{dal1993introduction}). The equicoercivity of the energy functionals is established in Proposition \ref{prop:equicoercive} and is based on the observation in \eqref{eq:helicalder}.
 In Propositions \ref{prop:liminf} and \ref{prop:limsup}, respectively,
 we prove that $\mathcal{F}_0$ provides a lower bound for the asymptotic behavior of the energies $(\mathcal{F}_{\varepsilon})_{\varepsilon>0}$ and show that this lower bound is indeed optimal.

Our second contribution concerns the derivation of a thin-film model for the LLG equation
which describes the magnetization dynamics in small ferromagnetic samples.
In our context, the LLG equation
reads (in strong form) as
\begin{equation}
 \begin{cases}  \partial_t \tmu  \eqs  - \tmu \times  \heffe[\tmu] + \alpha \tmu
    \times \partial_t \tmu  & \text{in } \mathcal{M} \times \RR_+,\\
    \partial_{\tmmathbf{n}} \tmu  \eqs  - \kappa \tmu \times \tmmathbf{n} &
    \text{on } \partial \mathcal{M} \times \RR_+,\\
    \tmu (0) \eqs  \tmu^0 & \text{in } \mathcal{M},
  \end{cases} \label{eq:LLGstronga}
\end{equation}
where $\alpha>0$ is a dimensionless damping factor,
while $\tmu^0$ is the magnetization at time $t=0$. The dynamics is driven by the effective field $\heffe[\tmu]:= -\partial_{\tmu}   \mathcal{F}_{\varepsilon}$ which is defined as the opposite of the first-order variation of the energy $\mathcal{F}_{\varepsilon}$. Our second result concerns  families  of weak solutions of the LLG equation
(we refer to Section \ref{sec:LLG} for the precise definition) and reads as follows.

\begin{theorem}
\label{thm:main2}
Let $\tmu^0(\sigma,s):=\tmu^0(\sigma)\in H^1(\omega,\Stwo^2)$ and, for every $\varepsilon>0$, let $\tmu_\varepsilon$ be a weak solution of~\eqref{eq:LLGstronga} with initial datum $\tmu^0$.
Then, there exists a magnetization $\tmu_0\in L^\infty(\RR_+;H^1(\omega,\Stwo^2))$ such that $\tmu_0\in H^1(\omega\times (0,T),\Stwo^2))$ for every $T>0$ and,  up to the extraction of a nonrelabeled subsequence,  
\[
\ue \rightharpoonup \tmu_0 \quad \text{weakly$^*$ in } L^\infty(\RR_+;H^1(\mathcal{M},\Stwo^2)).
\] 
The limit magnetization $\tmu_0$ satisfies in the weak sense the boundary value problem
\begin{equation}
  \begin{cases}
    \partial_t \tmu_0  \eqs - \tmu_0 \times \heffzero [ \tmu_0 ]
    + \alpha \tmu_0 \times \partial_t \tmu_0 & \text{in } \omega
    \times \RR_+,\\
    \partial_{\tmmathbf{n}} \tmu_0 \eqs - \kappa \tmu_0 \times
    \tmmathbf{n} & \text{on } \partial \omega \times \RR_+,\\
    \tmu_0 (0) \eqs  \tmu^0 & \text{in } \omega,
  \end{cases}  \label{eq:LLGlimitstrong}
\end{equation}
where the limiting effective field $\heffzero [ \tmu_0 ] \assign \Delta_{\omega} \tmu_0 - 2 \kappa
  \curl_{\omega} \tmu_0 - (1 + \kappa^2) (e_3 \otimes e_3) \tmu_0$
coincides with the opposite of the first-order variation of the $\Gamma$-limit energy $\mathcal{F}_0$ $($cf.~\eqref{eq:gammalimit}$)$. Moreover, for almost every  $T>0$, the following  energy  inequality holds
\begin{equation}
\label{eq:en-in}
\mathcal{F}_0(\tmu_0(T)) + \alpha \int_{0}^T \| \partial_t \tmu_0\|^2_{L^2(\omega,\RR^3)}\,\mathd t \leqslant \mathcal{F}_0(\tmu^0) .
\end{equation}  
\end{theorem}

The proof of this second result is based on the combination of some \textsl{a~priori} energy estimates with
{\cite[Lemma~2.1]{carbou2001thin}} (see also \cite[Lemma~1]{Di_Fratta_2020} for a more general statement),
and on the application of the Aubin--Lions lemma.

Our third and last result concerns the numerical approximation
of~\eqref{eq:LLGstronga} and~\eqref{eq:LLGlimitstrong}.
We propose an improved version of the projection-free finite element tangent plane scheme introduced
in~\cite{hrkac2017convergent}
(see Algorithm~\ref{alg:tps} below).
The main novelty consists in an implicit-explicit approach for treating
the effective field contributions,
which is designed in such a way that
the discrete energy law satisfied by the approximations
mimics the dissipative energy law of the continuous problem
(Proposition~\ref{prop:numerics}).
Moreover, suitable time reconstructions built using
the approximations generated by the algorithm converge weakly and unconditionally towards a weak solution
of the problem (Theorem~\ref{thm:numerics}).
Here,
the word `unconditionally' refers to the fact that the convergence analysis of Algorithm~\ref{alg:tps}
does not require
any coupling condition between the spatial mesh size and the time-step size.

Besides its own mathematical interest,
the derivation of effective thin-film micromagnetic models
also has relevant practical implications.
The capability to perform reliable micromagnetic simulations
using a 2D model is very favorable in terms of computational cost.
On the one hand, the complexity of the simulation is clearly reduced by the
dimension reduction (from 3D to 2D).
On the other hand,
since in the thin-film limit the nonlocal magnetostatic interaction
reduces to a local shape anisotropy,
one can save the cost of solving the magnetostatic Maxwell equations~\eqref{eq:FarMaxdemag}
(which usually involves linear algebra operations with large fully populated matrices).
This has a significant impact on the overall computational cost,
as, for practically relevant problem sizes,
the computation of the magnetostatic interaction is usually
the most time-consuming part of a micromagnetic simulation~\cite{aesds2013}.
In the last section of the present work,
presenting the results of two numerical experiments,
we validate our theoretical findings.
At the same time,
we show the effectivity of the proposed numerical scheme
and give a flavor of the computational benefits of our approach.

\subsection{Outline}

The paper is organized as follows.
In Section~\ref{subsec:equicoercivityffprime},
we prove the static $\Gamma$-convergence result of Theorem~\ref{thm:main1}.
In Section~\ref{sec:LLG}, we recall the definition of weak solutions of the LLG equation,
and provide a convergence analysis from the 3D setting to the reduced limiting  model.
Section~\ref{sec:num} is devoted to the description and the analysis of the approximation scheme,
and Section~\ref{sec:sim} to the presentation of the results of numerical simulations.

\section{Gamma-convergence: Proof of Theorem~\ref{thm:main1}}\label{subsec:equicoercivityffprime}

To prove Theorem~\ref{thm:main1}, we first show that the family $(\mathcal{F}_{\varepsilon})_{\varepsilon \in I}$ is
equicoercive in the weak topology of $H^1( \mathcal{M}, \Stwo^2)$. 
This step assures the validity of the
fundamental theorem of $\Gamma$-convergence concerning the variational
convergence of minimum problems ({\tmabbr{cf.}}
{\cite{braides1998homogenization,dal1993introduction}}).

\begin{proposition}
\label{prop:equicoercive}
There exists a (nonempty) weakly compact set $\mathcal{K} \subset H^1
( \mathcal{M}, \Stwo^2 )$ such that $\inf_{H^1 ( \mathcal{M},
\Stwo^2 )} \mathcal{F}_{\varepsilon} = \inf_\mathcal{K} \mathcal{F}_{\varepsilon}$
for every $\varepsilon \in I$.
\end{proposition}

\begin{proof}
Owing to {\eqref{eq:helicalder}}, the following identity holds true for every $\tmu \in H^1
( \mathcal{M}, \Stwo^2 )$:
\begin{equation}
  \mathcal{F}_{\varepsilon} ( \tmu ) + \kappa^2 | \mathcal{M} |
  = \frac{1}{2} \int_{\mathcal{M}} | \mathfrak{D}_{\varepsilon} \tmu
  |^2 \,\mathd x+ \frac{1}{2} \int_{\RR^3} | \Keps [ \tmu \chi_{\mathcal{M}}]
  |^2\, \mathd x, \label{eq:Functionalhel}
\end{equation}
where $\mathfrak{D}_{\varepsilon}$ is the
`Jacobian' matrix of the $\varepsilon$-rescaled
helical derivatives
\begin{equation}
  \mathfrak{D}_{\varepsilon} \tmu := \left(
    \partial_1 \tmu - \kappa ( e_1 \times \tmu ),
    \partial_2 \tmu - \kappa ( e_2 \times \tmu ),
    {\varepsilon}^{-1} \ds \tmu - \kappa ( e_3 \times \tmu )\right).
    \label{eq:jachelic}
\end{equation}
 If $\tmu \in H^1 ( \mathcal{M}, \Stwo^2 )$ is constant in
space then, taking into account {\eqref{eq:hdrels}}, we deduce the bound 
\begin{equation}
  \min_{\tmu \in H^1 ( \mathcal{M}, \Stwo^2 )}
  \mathcal{F}_{\varepsilon} ( \tmu ) \leqslant c_{\mathcal{M}},
\end{equation}
for some constant $c_{\mathcal{M}} > 0$ depending only on the volume of
$\mathcal{M}$. Therefore, for every $\varepsilon \in I$, the minimizers of
$(\mathcal{F}_{\varepsilon})_{\varepsilon \in I}$ are in the set
\begin{equation}
K (
\mathcal{M}, \Stwo^2 ) \assign \bigcup_{\varepsilon \in I} \{ \tmu \in
H^1 ( \mathcal{M}, \Stwo^2 ) \of \mathcal{F}_{\varepsilon} ( \tmu
) \leqslant \; c_{\mathcal{M}} \}.
\end{equation}
On the other hand, the
reverse Young inequality proves that for $i = 1, 2$
\begin{align}
  | \partial_i \tmu - \kappa ( e_i \times \tmu ) |^2 
  & \geqslant  \frac{1}{2} | \partial_i \tmu |^2 - \kappa^2, 
  \label{eq:ddf1}\\
  | {\varepsilon}^{-1} \ds \tmu - \kappa ( e_3 \times \tmu
  ) |^2 & \geqslant  \frac{1}{2 \varepsilon^2} | \ds \tmu
  |^2 - \kappa^2 .  \label{eq:ddf2}
\end{align}
In particular, for every $\tmu \in K ( \mathcal{M}, \Stwo^2 )$ we
have that $\| \grad \tmu \|^2_{L^2 ( \mathcal{M}, \RR^{3 \times 3}
)} \leqslant c_{\kappa, \mathcal{M}}$ for some positive constant
depending only on $\kappa$ and the measure of $\mathcal{M}$. Hence, $K (
\mathcal{M}, \Stwo^2 )$ is contained in a ball $B_{\mathcal{M}}$ of $H^1
( \mathcal{M}, \RR^3 )$ whose radius depends only on $\kappa$ and
$| \mathcal{M} |$. Setting $\mathcal{K} \assign B_{\mathcal{M}} \cap H^1
( \mathcal{M}, \Stwo^2 )$, we conclude that
\begin{equation}
  \inf_{H^1 ( \mathcal{M}, \Stwo^2 )} \mathcal{F}_{\varepsilon}
  = \inf_{\mathcal{K}} \mathcal{F}_{\varepsilon},
\end{equation}
where $\mathcal{K}$ is weakly compact being the intersection of the weakly closed
set $H^1 ( \mathcal{M}, \Stwo^2 )$ and the weakly compact set
$B_{\mathcal{M}}$.
\end{proof}

We proceed by showing that the functional $\mathcal{F}_0$ introduced in \eqref{eq:gammalimit} provides a lower bound for the asymptotic behavior of the energies $\mathcal{F}_{\varepsilon}$.

\begin{proposition}
\label{prop:liminf}
Let $(\ue)_{\varepsilon>0}$ in $H^1(\mathcal{M}, \Stwo^2)$ be such that 
$\liminf_{\varepsilon \rightarrow 0} \mathcal{F}_{\varepsilon} ( \ue )<+\infty$.
Then, there exist $\tmu_0 \in H^1 ( \omega, \Stwo^2)$ and $\tmmathbf{d}_0 \in L^2 ( \mathcal{M}, \RR^3)$ with
\begin{equation}
\tmmathbf{d}_0(\sigma,s)\cdot \tmu_0(\sigma) = 0 \quad \text{ for a.e. $(\sigma,s)\in  \mathcal{M}$},
\end{equation}
 such that, up to the extraction of a (not relabeled) subsequence, there holds
\begin{align}
  \qquad\qquad\qquad \ue (\sigma, s) & \, \rightarrow \, \tmu_0 (\sigma) \chi_I (s) & &
  \text{strongly in } L^2 ( \mathcal{M}, \Stwo^2 ) ,
  \label{eq:u0notsdep}\\
  \grad_{\omega}  \ue (\sigma, s) &  \, \rightharpoonup  \, \grad_{\omega} 
  \tmu_0 (\sigma) \chi_I (s) & & \text{weakly in } L^2 ( \mathcal{M},
  \RR^{2 \times 3} ), \\
  \varepsilon^{- 1} \ds  \ue (\sigma, s) &  \, \rightharpoonup  \,
  \tmmathbf{d}_0 (\sigma, s) & &\text{weakly in } L^2 ( \mathcal{M},
  \RR^3 ), \\
  \ds \ue (\sigma, s) &  \, \rightarrow  \, 0 & & \text{strongly in } L^2
  ( \mathcal{M}, \RR^3 ).\label{eq:ds} \qquad\qquad
\end{align}
Additionally, the following liminf inequality holds true
\begin{equation}
  \mathcal{F}_0 ( \tmu_0 ) \, \leqslant \, \liminf_{\varepsilon
  \rightarrow 0} \mathcal{F}_{\varepsilon} ( \ue ) .
  \label{eq:gammaliminf}
\end{equation}
\end{proposition}

\begin{proof}
Without loss of generality, we can assume that $\mathcal{F}_{\varepsilon} ( \ue ) \leqslant c_+<\infty$ for some constant $c_+ > 0$ and all $\varepsilon>0$. Given \eqref{eq:Functionalhel}, the estimates {\eqref{eq:ddf1}}--{\eqref{eq:ddf2}} give the uniform bound
\begin{equation}
  \sum_{i = 1}^2 \int_{\mathcal{M}} | \partial_i  \ue |^2\,\mathd x +
  \frac{1}{\varepsilon^2} \int_{\mathcal{M}} | \ds \ue
  |^2\,\mathd x +\| \Keps [ \ue \chi_{\mathcal{M}}] \|^2_{L^2 ( \RR^3, \RR^3 )}
\leqslant  \, c_{\kappa, \mathcal{M}}:= 2 (c_+ + \kappa^2 \vert \mathcal{M} \vert ) \,.
\end{equation}
Therefore, by weak compactness, we deduce the existence of $\tmu_0 \in H^1 ( \omega, \Stwo^2)$ and $ \tmmathbf{d}_0 \in L^2 ( \mathcal{M}, \RR^3)$ for which \eqref{eq:u0notsdep}--\eqref{eq:ds} hold true up to the extraction of a (not relabeled) subsequence.
In view of {\cite[Lemma~2.1]{carbou2001thin}}, there holds
\begin{equation}
  \Keps [ \ue \chi_{\mathcal{M}}]   \,\rightarrow  \, \tmmathbf{h}_0
  [ \tmu_0 ]  \quad \text{strongly in } L^2 (
  \RR^3, \RR^3 )  \label{eq:strayfieldlimit}
\end{equation}
with $\tmmathbf{h}_0 [ \tmu_0 ] \assign - \chi_{\mathcal{M}}  (e_3
\otimes e_3) \tmu_0$. Moreover, for $i = 1, 2$, we have that
\begin{align}
 \qquad\qquad\qquad \partial_i \ue - \kappa ( e_i \times \ue ) &  \,\rightharpoonup  \,\partial_i \tmu_0 - \kappa ( e_i \times \tmu_0 ) & &
  \text{weakly in } L^2 ( \mathcal{M}, \RR^3 ), 
  \label{eq:wconv1}\\
  \frac{1}{\varepsilon} \ds \ue - \kappa ( e_3 \times \ue ) &
  \, \rightharpoonup  \, \tmmathbf{d}_0 - \kappa ( e_3 \times \tmu_0
  ) & & \text{weakly in } L^2 ( \mathcal{M}, \RR^3 ). 
  \label{eq:wconv2}
\end{align}
As $\ue \rightarrow \tmu_0\chi_I $ strongly in $L^2 (
  \mathcal{M}, \Stwo^2 )$, we additionally obtain that
  \begin{equation}
  0 \eqs \varepsilon^{- 1}  \ds  \ue (\sigma, s) \cdot \ue
  (\sigma, s)  \; \rightharpoonup \; \tmmathbf{d}_0
  (\sigma, s) \cdot \tmu_0 (\sigma)
  \end{equation}
   weakly in $L^2 ( \mathcal{M}, \RR
  )$, and therefore $\tmmathbf{d}_0 (\sigma, s) \cdot \tmu_0 (\sigma)
  \eqs 0$ for almost every $(\sigma, s) \in \mathcal{M}$.
  
In terms of the helical derivatives \eqref{eq:helders}, the energy functional $\mathcal{F}_0$ reads as
\begin{equation}
  \mathcal{F}_0 ( \tmu_0 )=\frac{1}{2} \sum_{i = 1}^2
  \int_{\omega} | \held_i \tmu_0 (\sigma)  |^2 \mathd \sigma + \frac{1}{2} \int_{\omega}
  ( \tmu_0 (\sigma) \cdot e_3 )^2 \mathd \sigma  -\frac{\kappa^2}{2}|\omega|. 
  \label{eq:gammalimithel}
\end{equation}
Rewriting also the energy functional $\mathcal{F}_{\varepsilon}$ in terms of the helical derivatives
(see~\eqref{eq:Functionalhel}),
taking into account \eqref{eq:strayfieldlimit}--\eqref{eq:wconv2} and the weak lower semicontinuity of the norm,
the liminf inequality {\eqref{eq:gammaliminf}} follows at once.
\end{proof}

The next proposition shows that the lower bound identified in Proposition \ref{prop:liminf} is indeed optimal.

\begin{proposition}
\label{prop:limsup}
Let $\tmu_0 \in
H^1 ( \mathcal{M}, \Stwo^2 )$ be independent of the $s$-variable. Then, there exists a sequence $\{ \ues \}_{\varepsilon>0}\subseteq H^1(\mathcal{M}, \Stwo^2)$ such that $\ues \to \tmu_0$ strongly in $H^1 ( \mathcal{M}, \Stwo^2)$, and 
\begin{equation}
  \mathcal{F}_0 ( \tmu_0 ) = \, \lim_{\varepsilon
  \rightarrow 0} \mathcal{F}_{\varepsilon} ( \ues ) .
  \label{eq:limsup}
\end{equation}
\end{proposition}

\begin{proof}
For $\tmu_0 \in
H^1 ( \mathcal{M}, \Stwo^2 )$ independent of the $s$-variable, we
consider the nearest point projection on $\Stwo^2$ of a suitable perturbation
of $\tmu_0$ along the tangent space. Precisely, we set
\begin{equation} \label{eq:uesconstruction}
  \ues (\sigma, s) \assign \frac{\tmu_0 (\sigma) + \varepsilon s \kappa (
  e_3 \times \tmu_0 (\sigma) )}{| \tmu_0 (\sigma) + \varepsilon s
  \kappa ( e_3 \times \tmu_0 (\sigma) ) |} \quad\text{for almost every}\quad(\sigma,s)\in \mathcal{M}.
\end{equation}
We point out that $\ues$ is well defined since,
almost everywhere in $\mathcal{M}$,
$\tmu_0 \in \mathbb{S}^2$  and $(e_3 \times \tmu_0)\bot \tmu_0$
so that
$| \tmu_0 + s \varepsilon \kappa ( e_3 \times \tmu_0 ) |^2 \eqs 1 +
\varepsilon^2 s^2 \kappa^2 | e_3 \times \tmu_0 |^2 \geqslant 1$.
Note the pointwise convergence, $\ues\rightarrow \tmu_0$ as $\varepsilon \to 0$ for almost all $(\sigma,s)\in \mathcal{M}$. A direct computation shows that
\begin{eqnarray}
  \partial_1 \ues (\sigma, s) & \eqs & \frac{\ues \times ( (
  \partial_1 \tmu_0 + s \varepsilon \kappa e_3 \times \partial_1 \tmu_0
  ) \times \ues )}{| \tmu_0 + s \varepsilon
  \kappa ( e_3 \times \tmu_0 ) |} \;, \\[4pt]
  \partial_2 \ues (\sigma, s) & \eqs & \frac{\ues \times ( (
  \partial_2 \tmu_0 + s \varepsilon \kappa e_3 \times \partial_2 \tmu_0
  ) \times \ues )}{| \tmu_0 + s \varepsilon
  \kappa ( e_3 \times \tmu_0 ) |}, \\[4pt]
  \ds \ues (\sigma, s) & \eqs & \varepsilon \kappa \frac{\ues \times (
  ( e_3 \times \tmu_0 ) \times \ues )}{| \tmu_0 + s
  \varepsilon \kappa ( e_3 \times \tmu_0 ) |} . 
\end{eqnarray}
Hence, the dominated convergence theorem proves the following relations:
\begin{align}
  \qquad \qquad\qquad\ues (\sigma, s) &  \, \rightarrow  \, \tmu_0 (\sigma) \chi_I (s) &  \, 
 & \text{strongly in } L^2 ( \mathcal{M}, \Stwo^2 ), \\[3pt]
  \grad_{\omega}  \ues (\sigma, s) &  \, \to \, \grad_{\omega} 
  \tmu_0 (\sigma) \chi_I (s) &  \, & \text{strongly in } L^2 (
  \mathcal{M}, \RR^{2 \times 3} ),\qquad\\[3pt]
  \ds \ues (\sigma, s) &  \, \rightarrow  \, 0 &  \,  &\text{strongly in } L^2
  ( \mathcal{M}, \RR^3 ), \\[3pt]
  \varepsilon^{- 1} \ds  \ues (\sigma, s) &  \, \to  \, \kappa
  ( e_3 \times \tmu_0 ) & \, & \text{strongly in } L^2 (
  \mathcal{M}, \RR^3 ) .  \label{eq:weaktoke3m0}
\end{align}
In particular, we have that $\ues \rightarrow \tmu_0 \chi_I$
strongly in $H^1 ( \mathcal{M}, \Stwo^2 )$.  By \eqref{eq:Functionalhel}, \eqref{eq:gammalimithel}, and \cite[Lemma~2.1]{carbou2001thin}, we conclude that $\lim_{\varepsilon
\rightarrow 0} \mathcal{F}_{\varepsilon} ( \ues ) =\mathcal{F}_0
( \tmu_0 )$. This completes the proof of \eqref{eq:limsup} and of the proposition.
\end{proof}

We are finally in a position to prove Theorem \ref{thm:main1}.

\begin{proof}[Proof of Theorem \ref{thm:main1}]
The $\Gamma$-convergence statement of Theorem~\ref{thm:main1} follows by combining Propositions \ref{prop:liminf} and \ref{prop:limsup}.
It remains to prove that if the maps $\ue \in H^1(\mathcal{M}, \Stwo^2)$ are minimizers of $\mathcal{F}_{\varepsilon}$ then, upon possible extraction of a subsequence, $(\ue)_{\varepsilon>0}$ converges even strongly in $H^1(\mathcal{M}, \Stwo^2)$ to a minimizer of $\mathcal{F}_0$.

The convergence of $\ue$, weakly in $H^1(\mathcal{M}, \Stwo^2)$,  to a minimum point of $\mathcal{F}_0$ is a consequence of the equicoercivity of $\mathcal{F}_{\varepsilon}$ proved in Proposition \ref{prop:equicoercive}. Indeed, equicoercivity assures the validity of the fundamental theorem of $\Gamma$-convergence concerning the variational convergence of minimum problems ({\tmabbr{cf.}}
\cite{braides1998homogenization,dal1993introduction}).
Thus, we  only have to prove that the convergence to a minimum point is strong in $H^1(\mathcal{M}, \Stwo^2)$.
 
To this end we observe that, by assumptions, there exists $\tmu_0\in H^1(\mathcal{M}, \Stwo^2)$, not depending on the $s$ variable, such that $\ue \rightharpoonup \tmu_0$ weakly in $H^1(\mathcal{M}, \Stwo^2)$ and $\ds \ue (\sigma, s)   \to 0$ strongly in $L^2
  ( \mathcal{M}, \RR^3 )$. By the lower semicontinuity of the norm, and the strong convergence of $ \ue\to \tmu_0$ in $L^2( \mathcal{M}, \RR^3 )$, we get
\begin{align}
\mathcal{F}_0(\tmu_0) &\leqslant \liminf_{\varepsilon\to 0} \left\{\frac{1}{2} \int_{\mathcal{M}} |
  \nabla \ue |^2 \mathd x \notag  + \kappa
  \int_{\omega} \curl_{\omega}  \tmu_0 \cdot \tmu_0\,\mathd\sigma + \left(\frac{1 + \kappa^2}{2}\right)
  \int_{\omega} ( \tmu_0 (\sigma) \cdot e_3 )^2 \mathd \sigma-\frac{\kappa^2}{2}|\omega|\right\} \notag\\[4pt]
   &\leqslant \limsup_{\varepsilon\to 0} \left\{\frac{1}{2} \int_{\mathcal{M}} |
  \nabla \ue |^2 \mathd x \notag + \kappa
  \int_{\omega} \curl_{\omega}  \tmu_0 \cdot \tmu_0\mathd\sigma + \left(\frac{1 + \kappa^2}{2}\right)
  \int_{\omega} ( \tmu_0 (\sigma) \cdot e_3 )^2 \mathd \sigma-\frac{\kappa^2}{2}|\omega|\right\} \notag\\[4pt]
  & \eqs \limsup_{\varepsilon\to 0} \mathcal{F}_{\varepsilon}(\ue) \, \leqslant \, \lim_{\varepsilon\to 0} \mathcal{F}_{\varepsilon}(\ues) \eqs \mathcal{F}_0(\tmu_0),
\end{align}
where $\ues$ denotes the family  built from $u_0$  as in \eqref{eq:uesconstruction}. Note that, the last inequality is a consequence of the minimality of $\ue$, while the last equality is nothing but \eqref{eq:limsup}. Overall, we  conclude that 
\begin{equation}
\mathcal{F}_0(\tmu_0)=  \lim_{\varepsilon\to 0} \frac{1}{2} \int_{\mathcal{M}} |
  \nabla \ue |^2 \mathd x + \kappa
  \int_{\omega} \curl_{\omega}  \tmu_0 \cdot \tmu_0\mathd \sigma + \left(\frac{1 + \kappa^2}{2}\right)
  \int_{\omega} ( \tmu_0 (\sigma) \cdot e_3 )^2 \mathd \sigma-\frac{\kappa^2}{2}|\omega|, 
\end{equation}
from which  we deduce  the convergence of the norms $\|\ue\|_{H^1(\mathcal{M}, \Stwo^2)}\to \|\tmu_0\|_{H^1(\mathcal{M}, \Stwo^2)}$. Since $\ue \rightharpoonup \tmu_0$ weakly in $H^1(\mathcal{M}, \Stwo^2)$ we conclude that $\ue \to \tmu_0$ strongly in $H^1(\mathcal{M}, \Stwo^2)$.
\end{proof}

\section{The time-dependent case: Proof of Theorem \ref{thm:main2}} \label{sec:LLG}

The observable states of the magnetization at equilibrium correspond to the minimizers of the
micromagnetic energy functional {\eqref{eq:GLunormfull}} and they are among
the solutions of the weak Euler--Lagrange equation. Precisely, if $\m \in H^1
( \Omega, \Stwo^2 )$ is a metastable equilibrium state of
$\mathcal{G}_{\Omega}$, then
\begin{equation}
  \langle d\mathcal{G}_{\Omega} ( \m ), \tmmathbf{v}
  \rangle \eqs 0 \label{eq:BrSta}
\end{equation}
for every $\tmmathbf{v} \in H^1 ( \Omega, \RR^3 )$ such that
$\m\cdot \tmmathbf{v}=0$ a.e.~in $\Omega$, where $d\mathcal{G}_{\Omega}
: H^1 ( \Omega, \RR^3 ) \rightarrow \RR$ is the (unconstrained)
Frechet differential of $\mathcal{G}_{\Omega}$ at $\m$. Recall that $\heff [ \m
]=- d\mathcal{G}_{\Omega} ( \m )$.
Taking into account {\eqref{eq:helicalder}} with $\mathfrak{D}$ given by {\eqref{eq:jachelic}} for $\varepsilon =
1$, a simple computation
reveals that
\begin{equation}
  - \langle \heff [ \m ], \tmmathbf{\varphi} \rangle
  \eqs \int_{\Omega} \mathfrak{D} \m  \,\Fsp \,\mathfrak{D}\tmmathbf{\varphi}\,\mathd x-
  \int_{\Omega} \hd [\tmmathbf{m}\chi_{\mathcal{M}}] \cdot \tmmathbf{\varphi}\,\mathd x \quad\text{for all $\tmmathbf{\varphi} \in H^1 ( \Omega, \RR^3)$}\, .\label{eq:ELCc}
\end{equation}
 In particular, when $\m \in C^2 ( \bar{\Omega}, \Stwo^2 )$, we
can integrate by parts in {\eqref{eq:BrSta}} to obtain the strong form of Brown's static equation
\begin{equation}
\begin{cases}
    \m \times \heff [ \m ] \eqs 0 & \text{in } \Omega,\\
    \partial_{\tmmathbf{n}}  \m + \kappa ( {\m \times \tmmathbf{n}} )
    \eqs 0 & \text{on } \partial \Omega,
  \end{cases}
\end{equation}
where the effective field has the form $\heff [ \m ] \assign \Delta
\m - 2 \kappa^2 \m - 2 \kappa \curl  \m + \hd [ \m ]$, and
$\tmmathbf{n}$ is the outer normal unit vector field to $\partial \Omega$.

When the magnetization $\m$ does not satisfy Brown's static equation, the
ferromagnetic system is in a nonequilibrium state, and it evolves in time
according to the LLG equation
{\cite{LandauA1935,gilbert2004phenomenological}}, which, from the
phenomenological point of view, describes the magnetization dynamics as a
dissipative precession driven by the effective field. In strong form, the LLG equation reads as
\begin{equation}
 \begin{cases}
    \partial_t \m  \eqs  - \m \times \heff [ \m ] + \alpha \m
    \times \partial_t \m  & \text{in } \Omega \times \RR_+,\\
    \partial_{\tmmathbf{n}} \m  \eqs  - \kappa \m \times \tmmathbf{n} &
    \text{on } \partial \Omega \times \RR_+,\\
    \m (0) \eqs  \m^0 & \text{in } \Omega,
  \end{cases} \label{eq:LLGstrong}
\end{equation}
with $\alpha>0$ being the dimensionless Gilbert damping factor.

We recall below  the standard
weak formulation as described, {\tmabbr{e.g.}}, in
{\cite{alouges1992global,carbou2001thin,hrkac2017convergent}}. 

\begin{definition}
  Let $\Omega$ be an open set of $\RR^3$ and $\m^0 \in H^1 (
  \Omega, \Stwo^2 )$. For every $T >0$ we set $\Omega_T \assign
  (0, T) \times \Omega$. We say that the vector field $\m \in L^\infty(\RR_+;H^1 (\Omega , \Stwo^2))$ is a (global) weak solution of the LLG equation
  {\eqref{eq:LLGstrong}}, if the following
  conditions are fulfilled:
  \begin{itemize}
    \item For every $T>0$, the vector field $\m$ is in $H^1 ( \Omega_T, \Stwo^2 )$, and $\m (0) = \m^0$ in the trace sense.
    \item For every $T>0$ and every $\tmmathbf{\varphi} \in H^1 ( \Omega_T,^{} \RR^3
    )$, there holds
    \begin{equation}
      \int_0^T \duall{ \partial_t \m}{\tmmathbf{\varphi}}\,\mathd t \eqs  - \int_0^T
      \duall{\heff [ \m ]}{\tmmathbf{\varphi} \times \m}\,\mathd t + \alpha
      \int_0^T \duall{\partial_t \m}{\tmmathbf{\varphi} \times \m}\,\mathd t, 
      \label{eq:LLGweakform}
    \end{equation}
    where $\dual{\cdot}{\cdot}$ denotes the duality pairing associated with $H^1(\Omega,\RR^3)$.
    \item For almost all $T>0$, the following energy inequality holds:
    \begin{equation}
      \mathcal{G}_{\Omega} ( \m (T) ) + \alpha \int_0^T \|
      \partial_t \m (t) \|_{L^2 ( \Omega, \RR^3 )}^2 \, \mathd t
      \, \leqslant \, \mathcal{G}_{\Omega} ( \m^0 ) .
      \label{eq:energyinequality}
    \end{equation}
  \end{itemize}
\end{definition}

The evolution equation {\eqref{eq:LLGweakform}}  corresponds to  a weak formulation
of {\eqref{eq:LLGstrong}} in the space-time domain. The boundary conditions in
{\eqref{eq:LLGstrong}} are here enforced as natural boundary conditions.
Finally, the energy inequality {\eqref{eq:energyinequality}} is the weak
counterpart of the Gilbert dissipative energy law
\begin{equation}
  \partial_t \mathcal{G}_{\Omega} ( \m (t) ) \eqs - \alpha \|
  \partial_t \m (t) \|_{L^2 ( \Omega, \RR^3 )}^2 \leqslant 0,
\end{equation}
which is valid for  every  $t >0$ under suitable regularity assumptions on
the solution of {\eqref{eq:LLGstrong}}.

The existence of a global weak solution of {\eqref{eq:LLGstrong}} has recently
been proved in {\cite[Theorem~3.4]{hrkac2017convergent}}. In particular, in
our setting, for any $\varepsilon > 0$ there exists at least a solution
$\m_{\varepsilon}$ of the LLG equation on $\Omega_{\varepsilon}$ with initial
datum $\m^0_{^{}}$.

As in the proof of the static $\Gamma$-convergence result, we first proceed by
performing a space rescaling, and by rewriting the LLG equation on a fixed cylindrical
domain $\mathcal{M}= \omega \times I$, with $\omega$ being an open bounded domain in $\RR^2$ with Lipschitz boundary
and $I \assign (0, 1) \subseteq \RR$. First, for any $\ue \in H^1 (
\mathcal{M}, \Stwo^2 )$, we introduce the $\varepsilon$-rescaled
{\tmem{effective}} field, defined for every $\tmmathbf{\varphi} \in H^1 (
\mathcal{M}, \RR^3 )$ by
\begin{equation}
  - \langle \heffe [ \ue ], \tmmathbf{\varphi} \rangle 
  \eqs  \int_{\mathcal{M}} \mathfrak{D}_{\varepsilon}  \ue  \,\Fsp\,
  \mathfrak{D}_{\varepsilon} \tmmathbf{\varphi} \,\mathd x\, - \int_{\mathcal{M}} \Keps
  [ \ue \chi_{\mathcal{M}}] \cdot \tmmathbf{\varphi}\,\mathd x.  \label{eq:heffeps}
\end{equation}
We proceed by defining the rescaled weak solutions of the LLG equation.

\begin{definition}
\label{def:rescaled-wk-sol}
Let $\tmu^0(\sigma,s):=\tmu^0(\sigma)\in H^1(\omega,\Stwo^2)$ be independent of  the $s$-variable.
For every $T >0$ we set $\mathcal{M}_T \assign (0, T) \times \mathcal{M}$.
We say that the map $\ue \in L^\infty(\RR_+;H^1 (\mathcal{M} , \Stwo^2))$ is a
(global) weak solution (at scale $\varepsilon$) to the LLG equation
if the following conditions are fulfilled:
  \begin{itemize}
    \item For every $T > 0$, the vector field $\ue$ is in $ H^1 ( \mathcal{M}_T, \Stwo^2 )$, and $\ue (0) = \tmu^0_{^{}}$ in the trace sense.
    \item For every $T > 0$ and for every $\tmmathbf{\varphi} \in H^1 ( \mathcal{M}_T, \RR^3
    )$, there holds
    \begin{equation}
      \int_0^T \duall{\partial_t \ue}{\tmmathbf{\varphi}} \,\mathd t\, \eqs\, - \int_0^T
      \duall{\heffe [ \ue ]}{\tmmathbf{\varphi} \times \ue}\,\mathd t\, + \alpha
      \int_0^T \duall{\partial_t \ue}{\tmmathbf{\varphi} \times \ue}\,\mathd t , 
      \label{eq:LLGweakformrescaled}
    \end{equation}
    where $\dual{\cdot}{\cdot}$ denotes the duality pair in $H^{-1}(\mathcal{M},\RR^3) \times  H^1(\mathcal{M},\RR^3)$.
\item For almost every $T>0$, the following energy inequality holds:
    \begin{equation}
      \mathcal{F}_{\varepsilon} ( \ue (T) ) + \alpha \int_{0}^T \|
      \partial_t \ue (t) \|_{L^2 ( \mathcal{M}, \RR^3)}^2
      \,\mathd t \,\leqslant \mathcal{F}_{\varepsilon} ( \tmu^0_{^{}} )
      \label{eq:energyinequalityrescaled}\, ,
     \end{equation}
    with $\mathcal{F}_\varepsilon$ given by \eqref{eq:Functional}.
  \end{itemize}
\end{definition}

We are now in a position to prove Theorem \ref{thm:main2}. 

 \begin{proof}[Proof of Theorem \ref{thm:main2}]
Let $(\ue)_{\varepsilon}$ be a sequence of rescaled weak
solutions to the LLG equation, in the sense of Definition \ref{def:rescaled-wk-sol}.
By assumption, $\ue (0; \sigma,s) = \tmu^0(\sigma) \chi_I (s)$ for some
$\tmu^0 \in H^1 ( \omega, \Stwo^2 )$.
We first observe that, in view of \eqref{eq:hdrels}, there exists a positive
constant $\kappa_{\omega}$, depending only on the measure of $\mathcal{\omega}$
and on the DMI constant, such that, for every $\varepsilon > 0$, it holds that
\begin{equation}
  \mathcal{F}_{\varepsilon} ( \tmu^0 ) \;
  \leqslant \; \frac{1}{2} |\omega|+ \kappa_{\mathcal{\omega}}^2  \| \grad_{\omega} \tmu^0
  \|^2_{L^2 ( \omega, \RR^{2 \times 3} )} .
\end{equation}
Thus, due to the energy inequality {\eqref{eq:energyinequalityrescaled}}, there exists a positive constant $c_0$, depending only on the initial datum
$\tmu^0$, such that ${\textrm{ess sup}}_{t\in[0,T]} | \mathcal{F}_{\varepsilon}( \tmu_\varepsilon(t)| \leqslant c_0$ and $\|
      \partial_t \ue  \|^2_{ L^2 (\RR_+\times\mathcal{M}, \RR^3))} \leqslant c_0$ for every $\varepsilon > 0$ and for all $T>0$. In particular, the following uniform bounds hold:
\begin{align}
\qquad &  \| \nabla_{\omega} \ue \|_{L^{\infty} ( \RR_+; L^2 ( \mathcal{M},
  \RR^{2 \times 3} ) )}   \leqslant c_0,  
 & \frac{1}{\varepsilon} \| \partial_s \ue \|_{ L^{\infty} ( \RR_+ ; L^2(
  \mathcal{M},\RR^3) )}   \leqslant c_0, \qquad\qquad   \\[3pt]
 & \| \partial_t  \ue \|_{L^2 ( \RR_+ \times \mathcal{M},\RR^3 )}
  \leqslant c_0, \qquad\qquad &\| \Keps [ \ue ] \|_{L^{\infty} (
  \RR_+; L^2 ( \RR^3, \RR^3 ) )} \leqslant c_0. \qquad
  \end{align}
Hence, by weak compactness, there exist maps
$\tmmathbf{d}_0 \in
L^\infty ( \RR_+; L^2 ( \mathcal{M},\RR^3 ) )$, $\tmu_0 \in L^{\infty}
( \RR_+; H^1 ( \mathcal{M}, \Stwo^2 ))$, with $\tmu_0\in$ independent of the $s$ variable and belonging to $H^1(\omega_T,\Stwo^2)$ for every $T>0$,
such that, up to the extraction
of a (not relabeled) subsequence, the following convergence relations hold
true:
\begin{align}
  \qquad\qquad\qquad \ue & \overset{\ast}{\rightharpoonup}   \tmu_0
 & & \text{ weakly* in } L^{\infty} ( \RR_+; H^1 (
  \mathcal{M}, \RR^3 ) ),  \label{eq:m-ep-wk-time}\\[3pt]
  \Keps [ \ue ] & \overset{\ast}{\rightharpoonup} 
 -   (e_3 \otimes e_3) \tmu_0 \chi_{\mathcal{M}} & & \text{
  weakly* in } L^{\infty} ( \RR_+; L^2 ( \RR^3, \RR^3 )
  ),  \qquad\qquad \label{eq:K-ep-time}\\[3pt]
  \nabla_{\omega} \ue &\rightharpoonup 
  \nabla_{\omega} \tmu_0  & & \text{ weakly in } L^2 ( \mathcal{M}_T, \RR^{2 \times 3} ), 
  \label{eq:nabla-m-ep-wk-time}\\[3pt]
  \frac{1}{\varepsilon} \ds \ue & \rightharpoonup\tmmathbf{d}_0 & &\text{ weakly in } L^2 ( \mathcal{M}_T, \RR^{2 \times 3} ),  \label{eq:ds-m-ep-wk-time}\\[3pt]
  \partial_t \ue & \rightharpoonup \partial_t
  \tmu_0 &  & \text{ weakly in } L^2 ( \RR_+ \times \mathcal{M}, \RR^3
  ).  \label{eq:dt-m-ep-wk-time}
\end{align}
Here, we set $\omega_T:=\omega\times (0,T)$ and $\mathcal{M}_T:=\mathcal{M}\times(0,T)$.
Note that, {\eqref{eq:K-ep-time}} follows by {\cite[Lemma~2.1]{carbou2001thin}}, while the
fact that $\tmu_0$ is independent of the $s$-variable is a consequence of
{\eqref{eq:ds-m-ep-wk-time}}. Additionally, the Aubin-Lions lemma proves that
\begin{equation}
\ue  \rightarrow \tmu_0 \quad \text{strongly in } C^0(0,T;L^2 (\mathcal{M}, \Stwo^2) )  \label{eq:strong-m-ep}
\end{equation}
for every $T \in \RR_+$. This yields $\tmu_0 (t, \sigma) \in \Stwo^2$ for {\tmabbr{a.e.}} $(t,\sigma) \in \RR_+\times \omega$ and $\tmu_0 (0) = \tmu^0$.

Next, taking into account the saturation constraint $| \ue | = 1$
in $\mathcal{M}_T$, we infer, after some  direct  computations, that  for every $\tmmathbf{\varphi} \in
H^1 ( \mathcal{M}_T, \RR^3 )$ 
\begin{align}
  \mathfrak{D}_{\varepsilon}  \ue  \,\Fsp\, \mathfrak{D}_{\varepsilon} (
  \tmmathbf{\varphi} \times \ue ) & \eqs  \sum_{i=1}^3 \ue \times
  ( \partial^{\varepsilon}_i \ue - \kappa ( e_i \times \ue )
  ) \cdot \partial^{\varepsilon}_i \tmmathbf{\varphi}
  \nonumber\\
  &   \quad \quad \quad \quad \quad \quad + \kappa \sum_{i=1}^3 
  \{ ( \partial_i^{\varepsilon} \ue \cdot e_i ) \ue - (
  \ue \cdot e_i ) \partial^{\varepsilon}_i \ue \} \cdot
  \tmmathbf{\varphi},  \label{eq:limitllgexasym}
\end{align}
where $\partial_i^{\varepsilon} \assign \partial_i$ if $i = 1, 2$
and $\partial_3^{\varepsilon} \assign \varepsilon^{-1} \ds$.
Thus, multiplying both sides of
the previous relation by $\varepsilon$ and taking into account
{\eqref{eq:m-ep-wk-time}}--{\eqref{eq:strong-m-ep}}, we obtain that for
every $\tmmathbf{\varphi} \in H^1 ( \mathcal{M}_T,^{} \RR^3 )$
\begin{equation}
  \varepsilon \int_{\mathcal{M}_T} \mathfrak{D}_{\varepsilon}  \ue  \,\Fsp\,
  \mathfrak{D}_{\varepsilon} ( \tmmathbf{\varphi} \times \ue ) \,\mathd x\,\mathd t\,
  \xrightarrow{\varepsilon \rightarrow 0} \, \int_{\mathcal{M}_T} ( \tmu_0
  \times ( \tmmathbf{d}_0 - \kappa ( e_3 \times \tmu_0 )
  ) ) \cdot \ds \tmmathbf{\varphi}\,\mathd x\,\mathd t. 
\end{equation}
In particular, since $\ue$ is a weak solution of the LLG equation
{\eqref{eq:LLGweakformrescaled}}, multiplying both sides of
{\eqref{eq:LLGweakformrescaled}} by $\varepsilon$, taking into account
{\eqref{eq:m-ep-wk-time}}--{\eqref{eq:strong-m-ep}}, and then passing to
the limit for $\varepsilon \rightarrow 0$, we infer that
\begin{equation}
  \int_{\mathcal{M}_T} ( \tmu_0 \times ( \tmmathbf{d}_0 - \kappa
  ( e_3 \times \tmu_0 ) ) ) \cdot \ds
  \tmmathbf{\varphi}\,\mathd x\,\mathd t \eqs 0 \quad \forall \tmmathbf{\varphi} \in H^1 (
  \mathcal{M}_T,^{} \RR^3 ) \label{eq:d0equal} .
\end{equation}
Therefore, $\tmu_0 \times ( \tmmathbf{d}_0 - \kappa ( e_3 \times
\tmu_0 ) )$ is independent of the $s$-variable. Moreover, testing
{\eqref{eq:d0equal}} against functions of the form $\tmmathbf{\varphi} \assign
s\tmmathbf{\psi}$, with $\tmmathbf{\psi} \in H^1 ( \omega_T,^{} \RR^3
)$, we get, for {\tmabbr{a.e.}} $t \in [0, T]$, that
\begin{equation}
  \tmu_0 (t ; \sigma) \times ( \tmmathbf{d}_0 (t ; \sigma, s) - \kappa
  ( e_3 \times \tmu_0 (\sigma) ) \eqs 0 \quad \text{for
  {\tmabbr{a.e.}} } (\sigma, s) \in \mathcal{M}. \label{eq:rot-v-null}
\end{equation}
On the other hand, by the saturation constraint, we know that $0 =
\varepsilon^{- 1} \partial_s \ue \cdot \ue$ for every $\varepsilon > 0$. Thus,
by {\eqref{eq:ds-m-ep-wk-time}} and {\eqref{eq:strong-m-ep}}, we obtain that
\begin{equation}
  \tmu_0 \cdot \tmmathbf{d_0} = 0. \label{eq:scal-null}
\end{equation}
By combining {\eqref{eq:rot-v-null}} and {\eqref{eq:scal-null}} we conclude
that $\tmmathbf{d}_0$ is independent of the $s$-variable as well and, in fact,
\begin{equation}
  \tmmathbf{d_{0 \tmmathbf{}}} = \kappa ( e_3 \times \tmu_0 ) .
  \label{eq:d0-m0}
\end{equation}
Testing~\eqref{eq:limitllgexasym} with
$\tmmathbf{\varphi} \assign \tmmathbf{\psi} \chi_I$ where $\tmmathbf{\psi} \in
H^1 ( \omega_T,^{} \RR^3 )$, taking into account
{\eqref{eq:strong-m-ep}}, {\eqref{eq:nabla-m-ep-wk-time}}, and
{\eqref{eq:ds-m-ep-wk-time}}, and passing to the limit for $\varepsilon
\rightarrow 0$ we have
\begin{align*}
  \int_{\mathcal{M}_T} \mathfrak{D}_{\varepsilon}  \ue  \,\Fsp\,
  \mathfrak{D}_{\varepsilon} ( \tmmathbf{\varphi} \times \ue )\,\mathd x\,\mathd t 
  \xrightarrow{\varepsilon \rightarrow 0} &  \sum_{i=1}^2
  \int_{\omega_T} ( \tmu_0 \times \held_i \tmu_0 ) \cdot \partial_i
  \tmmathbf{\psi}\,\mathd \sigma\,\mathd t \\
  &   + \kappa \sum_{i=1}^2 \int_{\omega_T} \{ ( \partial_i
  \tmu_0 \cdot e_i ) \tmu_0 - ( \tmu_0 \cdot e_i ) \partial_i
  \tmu_0 \} \cdot \tmmathbf{\psi}\,\mathd \sigma\,\mathd t \\
  &   + \kappa \int_{\omega_T} \{ (\tmmathbf{d}_0 \cdot e_3) \tmu_0 -
  ( \tmu_0 \cdot e_3 ) \tmmathbf{d}_0 \} \cdot
  \tmmathbf{\psi}\,\mathd \sigma\,\mathd t. 
\end{align*}
In view of {\eqref{eq:d0-m0}}, we deduce that
\begin{align}
  \int_{\mathcal{M}_T} \mathfrak{D}_{\varepsilon}  \ue \, \Fsp\,
  \mathfrak{D}_{\varepsilon} ( \tmmathbf{\varphi} \times \ue )\,\mathd x\,\mathd t 
  \xrightarrow{\varepsilon \rightarrow 0} & \sum_{i=1}^2 \int_{\omega_T}
  \tmu_0 \times \held_i \tmu_0 \cdot \partial_i \tmmathbf{\psi}\,\mathd \sigma\,\mathd t \nonumber\\
  &   + \kappa \sum_{i=1}^2 \int_{\omega_T} \{
  ( \partial_i \tmu_0 \cdot e_i ) \tmu_0 - ( \tmu_0 \cdot e_i
  ) \partial_i \tmu_0 \} \cdot \tmmathbf{\psi}\,\mathd \sigma\,\mathd t \nonumber\\
  &  + \kappa^2
  \int_{\omega_T} ( \tmu_0 \times ( \tmu_0 \cdot e_3 ) e_3
  ) \cdot \tmmathbf{\psi}\,\mathd \sigma\,\mathd t \, . 
\end{align}
Next, we observe that, due to {\eqref{eq:K-ep-time}}, it holds
\begin{equation}
  - \int_{\mathcal{M}_T} \Keps [ \ue ] \cdot (
  \tmmathbf{\varphi} \times \ue )\,\mathd x\,\mathd t \xrightarrow{\varepsilon \rightarrow
  0} \int_{\omega_T} ( \tmu_0 (\sigma) \cdot e_3 ) ( \tmu_0
  \times e_3 ) \cdot \tmmathbf{\psi}\,\mathd \sigma\,\mathd t.
\end{equation}
Hence, passing to the limit in~\eqref{eq:LLGweakformrescaled},
we end up with the weak formulation of the
limiting LLG equation \eqref{eq:LLGlimitstrong} which reads, for every $\tmmathbf{\psi} \in H^1 (
\omega_T,^{} \RR^3 )$, as
\begin{equation}
  \int_0^T \duall{\partial_t \tmu_0}{\tmmathbf{\psi}}\,\mathd t  \eqs  - \int_0^T
  \duall{\heffzero [ \tmu_0 ]}{\tmmathbf{\psi} \times \tmu_0}\,\mathd t +
  \alpha \int_0^T \duall{\partial_t \tmu_0}{\tmmathbf{\psi} \times \tmu_0}\,\mathd t,
  \label{eq:LLGweakformrescaledlimit}
\end{equation}
where ({\tmabbr{cf.}} {\eqref{eq:heffeps}})
\begin{align}
  - \int_0^T \duall{\heffzero [ \tmu_0 ]}{\tmmathbf{\psi} \times
  \tmu_0}\,\mathd t & \eqs \sum_{i=1}^2 \int_{\omega_T} ( \tmu_0 \times
  \held_i \tmu_0 ) \cdot \partial_i \tmmathbf{\psi}\,\mathd \sigma\,\mathd t  \nonumber\\
  &  + (1 + \kappa^2)
  \int_{\omega_T} ( \tmu_0 \times ( \tmu_0 \cdot e_3 ) e_3
  ) \cdot \tmmathbf{\psi}\,\mathd \sigma\,\mathd t \nonumber\\
  &  + \kappa \sum_{i=1}^2 \int_{\omega_T} \{
  ( \partial_i \tmu_0 \cdot e_i ) \tmu_0 - ( \tmu_0 \cdot e_i
  ) \partial_i \tmu_0 \} \cdot \tmmathbf{\psi}\,\mathd \sigma\,\mathd t. 
\end{align}
This concludes the proof of \eqref{eq:LLGlimitstrong}.
To complete the proof of Theorem \ref{thm:main2} it remains to show \eqref{eq:en-in}.
To this end, we first observe that, since $\tmu_0\in L^{\infty} ( \RR_+; H^1 (\mathcal{M}, \RR^3 ) )$,
the function $t\mapsto \mathcal{F}_0(\tmu_0(t))$ is in the space $L^1(\RR_+)$, and thus almost every $t\in \RR_+$ belongs to the set $\mathcal{L}_0$ of its Lebesgue points. We claim that \eqref{eq:en-in} holds for every $t\in \mathcal{L}_0$. Indeed, fix $t\in \mathcal{L}_0$ and $\delta>0$.  Integrating \eqref{eq:energyinequalityrescaled} in time, for every $\varepsilon>0$ we deduce the inequality
  \begin{equation*}
      \frac1\delta\int_{t-\delta}^{t+\delta}\mathcal{F}_{\varepsilon} ( \ue (s) )\,\mathd s + \frac{\alpha}{\delta}\int_{t-\delta}^{t+\delta} \int_{0}^s \|
      \partial_t \ue (r) \|_{L^2 ( \mathcal{M}, \RR^3)}^2\,\mathd r\,\mathd s
      \leqslant \mathcal{F}_{\varepsilon} ( \tmu^0_{^{}} ).
           \end{equation*}     
Owing to~\eqref{eq:m-ep-wk-time}--\eqref{eq:strong-m-ep}, Fatou's lemma,
and {\cite[Lemma~2.1]{carbou2001thin}}, we obtain that
      \begin{equation*}
      \frac1\delta\int_{t-\delta}^{t+\delta}\mathcal{F}_{0} ( \tmu_0 (s) )\,\mathd s + \frac{\alpha}{\delta}\int_{t-\delta}^{t+\delta} \int_{0}^s \|
      \partial_t \tmu_0 (r) \|_{L^2 ( \mathcal{M}, \RR^3)}^2\,\mathd r\,\mathd s
      \leqslant \mathcal{F}_{0} ( \tmu^0_{^{}} ).
     \end{equation*}
Since $t$ is a Lebesgue point for $t\mapsto \mathcal{F}_0(\tmu_0(t))$,
property \eqref{eq:en-in} follows in the limit $\delta\to 0$.
\end{proof}

\clearpage

\section{Finite element discretization}
\label{sec:num}

In this section, we present a finite element method for the numerical approximation of the LLG equation
\begin{equation}
\label{eq:LLG-num}
\partial_t \m 
=
- \m \times \heff [ \m ]
+ \alpha \, \m \times \partial_t \m
\quad \text{in } U \times \RR_+,
\end{equation}
where $U \subset \RR^d$ ($d=2,3$) is a bounded Lipschitz domain
with polytopal boundary,
$\alpha>0$,
and the effective field $\heff [ \m ] = - d \mathcal{G}_U[\m]$
is obtained from the energy functional
\begin{equation} \label{eq:functional-num}
\mathcal{G}_U(\m)
= \frac{1}{2} \int_U \vert \nabla \m \vert^2
+ \kappa \int_U \curl\m \cdot \m
- \frac{1}{2} \int_U \pi[\m] \cdot \m
- \int_U \f \cdot \m.
\end{equation}
Here,
$\nabla(\cdot)$ and $\curl(\cdot)$ denote
$d$-dimensional realizations of the gradient and curl operators
(see Remark~\ref{rem:curl2D} for the definition of the 2D curl),
$\kappa \in \RR$,
$\f \in L^2(U,\RR^3)$,
and
$\pi : L^2(U,\RR^3) \to L^2(U,\RR^3)$ is a linear, bounded, and self-adjoint operator.
Note that this setting covers both the 3D model
(see~\eqref{eq:GLunormfull} and~\eqref{eq:LLGstrong},
where $U = \Omega \subset \RR^3$, $\pi[\m] = \hd [\m \chi_{\Omega} ]$, and $\f \equiv \tmmathbf{0}$)
and the 2D reduced model
(see~\eqref{eq:gammalimit} and~\eqref{eq:LLGlimitstrong},
where $U = \omega \subset \RR^2$, $\pi[\m] = - (1 + \kappa^2) (e_3 \otimes e_3) \m$, and $\f \equiv \tmmathbf{0}$).
To complete the setting, \eqref{eq:LLG-num} is supplemented with
natural boundary conditions on $\partial U \times \RR_+$
(which, given the functional in~\eqref{eq:functional-num},
read $\partial_{\tmmathbf{n}} \m =  - \kappa \, \m \times \tmmathbf{n}$)
and the initial condition $\m (0) = \m^0$ in $U$.

For the numerical approximation of~\eqref{eq:LLG-num},
we propose a tangent plane scheme
in the spirit of~\cite{alouges2008,multiscale2014,abertetal2014,hrkac2017convergent}.
The basic ingredients of the method are:
\begin{itemize}
\item for the time discretization,
a uniform subdivision of $\RR_+$ with constant time-step size $\tau>0$
i.e., $t_i := i \tau$ for all $i \in \NN_0$;
\item for the spatial discretization,
a quasi-uniform family 
$( \mathcal{T}_h )_{h>0}$ of regular simplicial meshes of $U$
parametrized by the mesh size $h>0$.
\end{itemize}
We denote by $\mathcal{N}_h$ the set of vertices of $\mathcal{T}_h$.
For each simplex $K \in \mathcal{T}_h$ (a triangle if $d=2$, a tetrahedron if $d=3$),
we denote by $\mathcal{P}^1(K)$ the space of first-order polynomials on $K$.
We consider the space
of piecewise affine and globally continuous functions from $U$ to $\RR$,
i.e.,
\begin{equation*}
\mathcal{S}^1(\mathcal{T}_h)
= \left\{v_h \in C^0(\overline{U}): v_h \vert_K \in \mathcal{P}^1(K) \text{ for all } K \in \mathcal{T}_h \right\},
\end{equation*}
and we denote by
$\mathcal{I}_h : C^0(\overline{U}) \to \mathcal{S}^1(\mathcal{T}_h)$
the associated nodal interpolant, i.e.,
\begin{equation*}
\mathcal{I}_h [v](z) = v(z)
\quad \text{for all }
v \in C^0(\overline{U})
\text{ and all }
z \in \mathcal{N}_h.
\end{equation*}
To mimic the intrinsic orthogonality property of the LLG equation (i.e., $\m \cdot \partial_t \m = 0$)
at the discrete level, for all $\ppsi_h \in \mathcal{S}^1(\mathcal{T}_h)^3$,
we consider the \emph{discrete tangent space} of $\ppsi_h$, defined by
\begin{equation} \label{eq:discreteTangentSpace}
\mathcal{K}_h(\ppsi_h)
:= \left\{\pphi_h \in \mathcal{S}^1(\mathcal{T}_h)^3 : \ppsi_h(z) \cdot \pphi_h(z) = 0 \text{ for all } z \in \mathcal{N}_h \right\},
\end{equation}
where the orthogonality is enforced only at the vertices of the mesh.

The starting point of the method is the equivalent reformulation of~\eqref{eq:LLG-num}
in the form
\begin{equation} \label{eq:LLG-alternative}
\alpha \, \partial_t \m
+ \m \times \partial_t \m
= \heff [ \m ] - ( \heff [ \m ] \cdot \m ) \m,
\end{equation}
which is linear in the unknown $\v := \partial_t\m$.
For each $i \in \NN_0$, given an approximation $\m_h^i \approx \m(t_i)$,
the idea is
to compute $\v_h^i \approx \v(t_i) = \partial_t\m(t_i)$
by means of a weak formulation of~\eqref{eq:LLG-alternative}
posed on the discrete tangent space $\mathcal{K}_h(\m_h^i)$,
and then
to use $\v_h^i$ to update $\m_h^i$ to $\m_h^{i+1}$ via a first-order time-stepping.
In the following algorithm, we state the proposed method for the numerical approximation
of~\eqref{eq:LLG-num}.

\begin{algorithm} \label{alg:tps}
Input:
$\m_h^0 \in \mathcal{S}^1(\mathcal{T}_h)^3$ satisfying $\vert \m_h^0(z)\vert = 1$ for all $z \in \mathcal{N}_h$. \\
Loop:
For all $i \in \NN_0$, iterate:
\begin{itemize}
\item[\rm(i)]  Compute $\v_h^i \in \mathcal{K}_h(\m_h^i)$ such that,
for all $\pphi_h \in \mathcal{K}_h(\m_h^i)$, it holds that
\begin{equation} \label{eq:tps}
\begin{split}
&\qquad\qquad \alpha \int_U \mathcal{I}_h[\v_h^i \cdot \pphi_h]
+ \int_U \mathcal{I}_h[(\m_h^i \times \v_h^i) \cdot \pphi_h]
+ \tau \int_U \nabla \v_h^i : \nabla\pphi_h
\\
&
\qquad\qquad\qquad\qquad\qquad\qquad\qquad\qquad\qquad+ \frac{\kappa \tau}{2} \int_U \curl\v_h^i \cdot \pphi_h
+ \frac{\kappa \tau}{2} \int_U \v_h^i \cdot \curl\pphi_h
\\
& \quad\qquad = - \int_U \nabla\m_h^i : \nabla\pphi_h
+ \int_U \pi[\m_h^i] \cdot \pphi_h
+ \int_U \f \cdot \pphi_h
\\
&\qquad\qquad\qquad\qquad\qquad\qquad\qquad\qquad\qquad
- \kappa \int_U \curl\m_h^i \cdot \pphi_h
- \kappa \int_U \m_h^i \cdot \curl\pphi_h.
\end{split}
\end{equation}
\item[\rm(ii)] Define $\m_h^{i+1} \in \mathcal{S}^1(\mathcal{T}_h)^3$
by
\begin{equation} \label{eq:tps2}
\m_h^{i+1} := \m_h^i + \tau \v_h^i.
\end{equation}
\end{itemize}
Output:
Sequences of approximations
$(\m_h^i)_{i \in \NN_0}$
and $(\v_h^i)_{i \in \NN_0}$.
\end{algorithm}

Let us state the main properties of Algorithm~\ref{alg:tps}.
The variational formulation~\eqref{eq:tps} is a Galerkin approximation of~\eqref{eq:LLG-alternative},
where the last term in~\eqref{eq:LLG-alternative}, parallel to $\m$,
does not appear, because~\eqref{eq:tps} is posed in the discrete tangent space of $\m_h^i$.
The time discretization adopts an implicit-explicit approach,
where the effective field contributions involving spatial derivatives
are treated implicitly
(the exchange contribution by the implicit Euler method,
the DMI contribution by the Crank--Nicolson method),
while the lower-order contributions collected in the operator $\pi[\cdot]$ and in $\f$ are treated explicitly.
The presence of the nodal interpolant in the first two integrals on the left-hand side of~\eqref{eq:tps}
means that these are computed using the trapezoidal rule
(so-called \emph{mass lumping}).
Finally, we observe that,
if the time-step size $\tau$ is sufficiently small,
the bilinear form on the left-hand side of~\eqref{eq:tps} is elliptic in $H^1(U,\RR^3)$.
Hence, each step of Algorithm~\ref{alg:tps} is well-defined.

\begin{remark}
The evaluation of the quantity $\pi[\m_h^i]$ can involve the solution of a PDE
(which is the case, e.g., when $\pi[\m] = \hd [\m \chi_{\Omega}]$),
for which an approximation method is also required in general.
Here, for ease of presentation, we neglect this aspect 
and refer, e.g., to~\cite{multiscale2014} for an analysis which accounts for the inexact evaluation of $\pi[\cdot]$
(in terms of an approximate operator $\pi_h[\cdot]$).
\end{remark}

Weak solutions to the LLG equation satisfy the unit-length constraint $\vert \m \vert = 1$
and a dissipative energy law;
see, e.g., the energy inequality~\eqref{eq:energyinequality}.
In the following proposition,
we establish discrete counterparts of these properties satisfied by the
iterates of Algorithm~\ref{alg:tps}.

\begin{proposition} \label{prop:numerics}
For all $i \in \NN_0$,
the iterates of Algorithm~\ref{alg:tps} satisfy the discrete energy law 
\begin{equation} \label{eq:energy_decay}
\mathcal{G}_U(\m_h^{i+1})
+ \alpha \tau \int_U \mathcal{I}_h\big[\lvert \v_h^i \rvert^2\big]
+ \frac{\tau^2}{2} \int_U \lvert \nabla \v_h^i \rvert^2
+ \frac{\tau^2}{2} \int_U \pi[\v_h^i] \cdot \v_h^i
= \mathcal{G}_U(\m_h^i).
\end{equation}
In particular, it holds that $\mathcal{G}_U(\m_h^{i+1}) \leqslant \mathcal{G}_U(\m_h^i)$.
Moreover, for all $N \in \NN$, it holds that
\begin{equation} \label{eq:constraint_error}
\Vert \mathcal{I}_h\big[\vert \m_h^N\vert^2\big] - 1 \Vert_{L^1(U)}
\leqslant C \tau^2 \sum_{i = 0}^{N-1}\int_U \mathcal{I}_h\big[\lvert \v_h^i\rvert^2\big],
\end{equation}
where the constant $C>0$ depends only on the shape-regularity of the family of meshes.
\end{proposition}

\begin{proof}
Let $i \in \NN_0$.
Choosing the test function $\pphi_h = \v_h^i \in \mathcal{K}_h(\m_h^i)$ in~\eqref{eq:tps},
we obtain the identity
\begin{align} \label{eq:tps_aux}
& \alpha \int_U \mathcal{I}_h\big[\lvert \v_h^i\rvert^2\big]
+ \tau  \int_U \lvert \nabla \v_h^i \rvert^2
+ \kappa \tau \int_U \curl\v_h^i \cdot \v_h^i
\\
\notag
& \quad = - \int_U \nabla\m_h^i : \nabla\v_h^i
- \kappa \int_U \curl\m_h^i \cdot \v_h^i
- \kappa \int_U \m_h^i \cdot \curl\v_h^i
+ \int_U \pi[\m_h^i] \cdot \v_h^i
+ \int_U f \cdot \v_h^i.
\end{align}
Exploiting the fact that $\pi[\cdot]$ is self-adjoint.
It follows that
\begin{equation*}
\begin{split}
\mathcal{G}_U(\m_h^{i+1})
\stackrel{\eqref{eq:functional-num}}{=}& 
\frac{1}{2} \int_U \vert \nabla \m_h^{i+1} \vert^2
+ \kappa \int_U \curl\m_h^{i+1} \cdot \m_h^{i+1}
- \frac{1}{2} \int_U \pi[\m_h^{i+1}] \cdot \m_h^{i+1}
- \int_U \f \cdot \m_h^{i+1} \\
\stackrel{\eqref{eq:tps2}}{=}& 
\mathcal{G}_U(\m_h^i)
+ \tau  \int_U \nabla\m_h^i : \nabla\v_h^i
+ \frac{\tau^2}{2} \int_U \lvert \nabla \v_h^i \rvert^2 \\
& 
+ \kappa \tau \int_U \curl\m_h^i \cdot \v_h^i
+ \kappa \tau \int_U \curl\v_h^i \cdot \m_h^i 
+ \kappa \tau^2 \int_U \curl\v_h^i \cdot \v_h^i \\
& 
- \frac{\tau}{2} \int_U \pi[\m_h^i] \cdot \v_h^i
- \frac{\tau}{2} \int_U \pi[\v_h^i] \cdot \m_h^i
- \frac{\tau^2}{2} \int_U \pi[\v_h^i] \cdot \v_h^i
- \tau \int_U \f \cdot \v_h^i
\\
\stackrel{\eqref{eq:tps_aux}}{=}& 
\mathcal{G}_U(\m_h^i)
- \alpha \tau \int_U \mathcal{I}_h\big[\lvert \v_h^i\rvert^2\big]
- \frac{\tau^2}{2} \int_U \lvert \nabla \v_h^i \rvert^2
- \frac{\tau^2}{2} \int_U \pi[\v_h^i] \cdot \v_h^i.
\end{split}
\end{equation*}
This yields~\eqref{eq:energy_decay}.
Moreover, observing that the last three terms on the left-hand side of~\eqref{eq:energy_decay} are nonnegative,
we also conclude that $\mathcal{G}_U(\m_h^{i+1}) \leqslant \mathcal{G}_U(\m_h^i)$.

Let $z \in \mathcal{N}_h$.
Since $ \v_h^i \in \mathcal{K}_h(\m_h^i)$, from~\eqref{eq:tps2}, we deduce that
$\lvert \m_h^{i+1} (z) \rvert^2 = \lvert \m_h^i (z) \rvert^2 + \tau^2 \lvert \v_h^i (z) \rvert^2$.
Let $N \in \NN$.
We infer that
\begin{equation*}
\lvert \m_h^N (z) \rvert^2 - 1
=
\lvert \m_h^N (z) \rvert^2 - \lvert \m_h^0 (z) \rvert^2
=
\sum_{i = 0}^{N-1} \left( \lvert \m_h^{i+1} (z) \rvert^2 - \lvert \m_h^i (z) \rvert^2 \right)
=
\tau^2 \sum_{i = 0}^{N-1} \lvert \v_h^i (z) \rvert^2.
\end{equation*}
With this identity,
using the norm equivalences from~\cite[Lemma~3.4 and~3.9]{bartels2015},
we obtain~\eqref{eq:constraint_error}.
\end{proof}

Proposition~\ref{prop:numerics} shows that Algorithm~\ref{alg:tps} respects the dissipative dynamics
of the LLG equation,
where the intrinsic energy dissipation
(modulated by the damping parameter $\alpha$,
cf.\ the second term on the left-hand side of~\eqref{eq:energy_decay})
is augmented by some artificial dissipation associated with the
implicit-explicit treatment of the exchange contribution and the $\pi$-contribution
to the effective field;
cf.\ the nonnegative third and the fourth term on the left-hand side of~\eqref{eq:energy_decay}.
Unlike that and unlike the integrator from~\cite{hrkac2017convergent}, the DMI contribution,
now treated by the symplectic Crank--Nicolson method,
does not generate numerical dissipation
(which would not have a definite sign, as the DMI energy does not have one).
Moreover, although the scheme does not enforce the unit-length
constraint on the approximate magnetization (not even at the vertices of the mesh),
the violation can be controlled by the time-step size;
cf.\ \eqref{eq:constraint_error}.

Algorithm~\ref{alg:tps} yields a sequence of discrete functions $(\m_h^i)_{i \in \NN_0}$,
which can be used to define the piecewise affine time reconstruction $ \m_{h\tau}$ defined by
\begin{equation*}
\m_{h\tau}(t) := \frac{t-t_i}{\tau}\m_h^{i+1} + \frac{t_{i+1} - t}{\tau}\m_h^i
\quad
\text{for all } i \in \NN \text{ and } t \in [t_i,t_{i+1}).
\end{equation*}
The following theorem establishes the convergence of the sequence $(\m_{h\tau})_{h,\tau>0}$ of time reconstructions
towards a weak solution of the LLG equation as $h,\tau \to 0$.

\begin{theorem} \label{thm:numerics}
Let the approximate initial condition satisfy the convergence property
\begin{equation*}
\m_h^0 \to \m^0 \quad \text{in } H^1(U,\RR^3) \quad \text{as } h \to 0.
\end{equation*}
Then, there exist a weak solution $\m$ of the LLG equation
and a subsequence of $(\m_{h\tau})_{h,\tau>0}$
which,
for all $T>0$,
converges weakly und unconditionally towards $\m$
in $H^1((0,T) \times U,\RR^3)$
as $h,\tau \to 0$.
\end{theorem}

The proof of Theorem~\ref{thm:numerics} is constructive and is based on a standard compactness argument:
Starting from the energy estimate~\eqref{eq:energy_decay}, one can show that the sequence $(\m_{h\tau})_{h,\tau>0}$
is uniformly bounded in $L^{\infty}(\RR_+;H^1(U,\RR^3)) \cap H^1(U \times (0,T),\RR^3)$ (for all $T>0$),
which allows extracting a weakly convergent subsequence.
Its limit is then identified with a weak solution of the LLG equation by passing to the limit as $h,\tau \to 0$
in the discrete variational formulation~\eqref{eq:tps}, in the energy law~\eqref{eq:energy_decay},
as well as in~\eqref{eq:constraint_error}.
We omit the details as the argument follows the ideas in~\cite{alouges2008,multiscale2014,abertetal2014,hrkac2017convergent}.
Note that a byproduct of the constructive proof is an existence result for weak solutions of~\eqref{eq:LLG-num}.

\section{Numerical results}
\label{sec:sim}

In this section,
we aim to highlight the practical implications of our results
and show the effectivity of the proposed algorithm
by means of two numerical experiments.
The computations presented in this section have been performed
with the open-source micromagnetic software Commics~\cite{prsehhsmp2020}.

For both experiments,
the computational domain is a helimagnetic nanodisk of thickness \SI{9}{\nano\meter}
(aligned with the $x_3$-direction)
and variable diameter $d>0$ (aligned with the $x_1 x_2$-plane),
i.e.,
$\Omega = \omega \times (0,9)$,
with $\omega = \{ x \in \RR^2 : \lvert x \rvert < d/2 \}$
being a circle with diameter $d$.
We consider the energy functional in physical units (with values measured in~\si{\joule})
\begin{equation} \label{eq:energy_physical}
\mathcal{E}(\m)
=
A \int_\Omega \lvert \nabla \m \rvert^2
+
D \int_{\Omega} \curl  \m \cdot \m
-
\frac{\mu_0 M_s^2}{2} \int_{\Omega}
 \hd [ \m \chi_{\Omega} ] \cdot \m.
\end{equation}
Here, $A>0$ is the exchange stiffness constant
(in \si{\joule\per\meter}),
$D \in \RR$ is the DMI constant (in~\si{\joule\per\square\meter}),
while $M_s>0$ is the saturation magnetization (in~\si{\ampere\per\meter}).
In our experiments,
we use the material parameters of iron-germanium (FeGe), i.e.,
$A=$ \SI{8.78e-12}{\joule\per\meter}, $D=$ \SI{1.58e-3}{\joule\per\square\meter}, and $M_s=$ \SI{3.84e5}{\ampere\per\meter};
see, e.g., \cite{bcwcvbachsf2015,babcwcvhcsmf2017}.
The resulting exchange length is $\ell_{\mathrm{ex}} = \sqrt{2A / (\mu_0 M_s^2)} \approx$ \SI{9.73}{\nano\meter}.
Note that the energy functional~\eqref{eq:energy_physical} can be rewritten in nondimensional form as~\eqref{eq:GLunormfull}
after a suitable rescaling of the spatial variable and the involved material parameters
($x' = x / \ell_{\mathrm{ex}}$, $\kappa = D / (\mu_0 M_s^2 \ell_{\mathrm{ex}})$,
$\mathcal{G}_\Omega(\m)=\mathcal{E}(\m) / (\mu_0 M_s^2 \ell_{\mathrm{ex}}^3)$).

\subsection{Comparison of the models in the stationary case} \label{sec:exp1}

In our first numerical experiment, which is inspired by~\cite{bcwcvbachsf2015},
we investigate the diameter influence on the equilibrium magnetization configurations
obtained by relaxing a uniform out-of-plane ferromagnetic state.
Performing this study,
we compare the results obtained with the full 3D model
and the reduced 2D model.

We consider the disk diameters
$d=$ \num{80}, \num{90}, \num{100}, \num{120}, \num{140}, \num{160}, \num{180}, \SI{200}{\nano\meter}.
For the
time discretization in Algorithm~\ref{alg:tps},
we use a constant time-step size of \SI{1e-11}{\pico\second}.
For the spatial discretization of each nanodisk---a cylinder in 3D (resp., a circle in 2D)---we
consider a tetrahedral (resp., triangular) mesh
with mesh size of about $h_{3D}=$ \SI{5.25}{\nano\meter} in 3D
(resp., about $h_{2D}=$ \SI{4.45}{\nano\meter});
see Table~\ref{tab:exp1_results} below for the precise values.
These values, which are well below the exchange length of the material, are chosen so that
the surface mesh of the top face of the 3D nanodisk
has approximately the same mesh size as the mesh of the 2D nanodisk.

For each considered value of the disk diameter,
we start from the uniform out-of-plane initial condition
$\m^0 \equiv (0,0,1)$
and
let the LLG dynamics evolve the system towards its equilibrium.
Since we are not interested in the precise magnetization dynamics,
to speed up the simulations,
we choose the large value $\alpha=$ \num{1} for the Gilbert damping constant in the LLG equation.
We simulate
for \SI{1}{\nano\second},
which experimentally turns out to be a sufficiently large time to reach
the stable state for all diameters.

\begin{figure}[t]
\includegraphics[width=16.5cm]{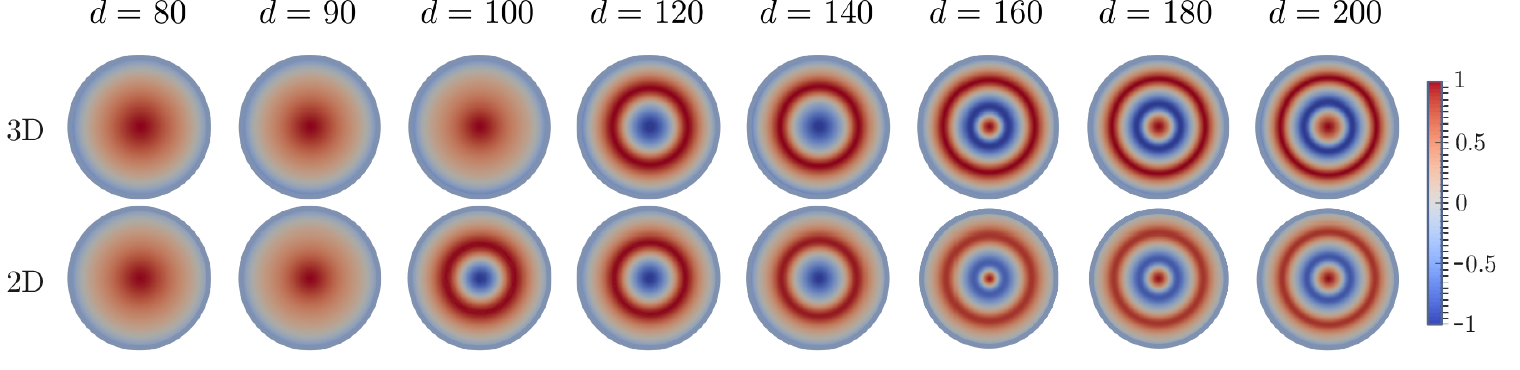}
\caption{Experiment of Section~\ref{sec:exp1}:
Out-of-plane magnetization component $m_3$ of the equilibrium state
for different values of the disk diameter $d$ (in~\si{\nano\meter})
computed with the
full 3D model (top) and
reduced 2D model (bottom).
}
\label{fig:experiment1_stable}
\end{figure}

In Figure~\ref{fig:experiment1_stable}, we plot
the out-of-plane magnetization component $m_3$ of the equilibrium state
for all considered values of the disk diameter
and both the full 3D model
and the reduced 2D model.

\begin{figure}[t]
\centering
\includegraphics[width=16.4cm]{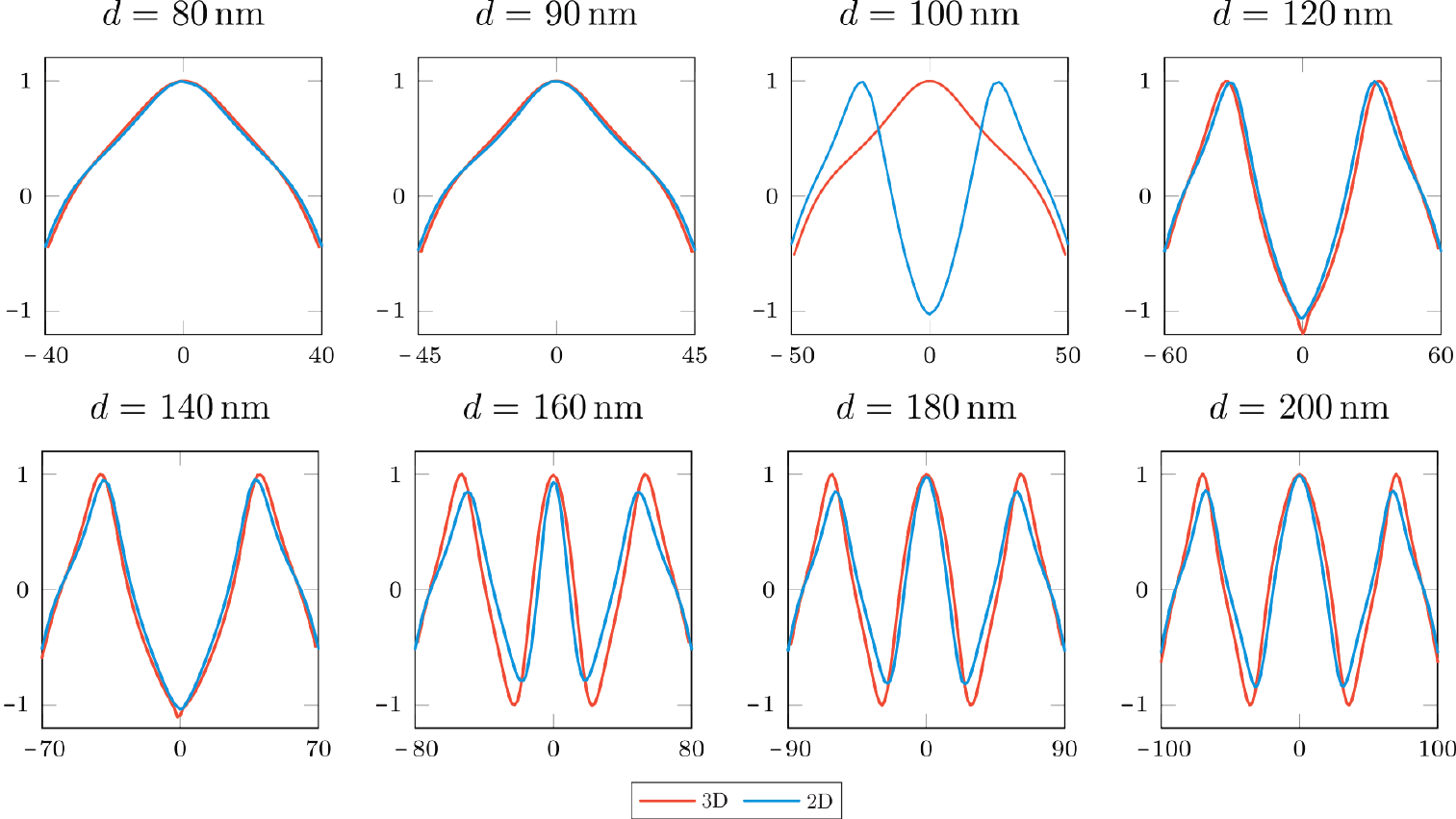}
\caption{Experiment of Section~\ref{sec:exp1}:
Out-of-plane magnetization component $m_3$ along one horizontal symmetry line
for different values of the disk diameter.
Comparison of the results obtained with the full 3D model and the reduced 2D model.}
\label{fig:experiment1_profiles}
\end{figure}

In Figure~\ref{fig:experiment1_profiles},
for all considered values of the disk diameter,
we plot $m_3$ along one horizontal symmetry line
for both the full 3D model and the reduced 2D model.
For the 3D case,
we consider the line between the points $(-d/2,0,0)$ and $(d/2,0,0)$.
For the 2D case,
we consider the line between the points $(-d/2,0)$ and $(d/2,0)$.

Figures~\ref{fig:experiment1_stable}--\ref{fig:experiment1_profiles}
show the influence of the diameter of the nanodisk on the final equilibrium state.
All equilibrium configurations are radially symmetric.
For the smallest values of $d$ considered in the experiment,
the out-of-plane magnetization component $m_3$
along one horizontal symmetry line does not cover the entire interval $[-1,1]$.
Following the terminology of~\cite{bcwcvbachsf2015},
we call this state an \emph{incomplete skyrmion}.
Increasing the diameter,
the magnetization relaxes to a different stable state,
in which $m_3$ covers the entire interval $[-1,1]$ at least once.
We refer to this state as an \emph{isolated skyrmion}.
The transition between the two states occurs for a diameter between
\num{100} and \SI{120}{\nano\meter} for the full 3D model,
while it occurs for a diameter below \SI{100}{\nano\meter}
for the reduced 2D model.
For both models, for the cases $d=$ \num{160}, \num{180}, \SI{200}{\nano\meter},
$m_3$ covers the entire interval $[-1,1]$ at least twice.
We refer to this state as a \emph{target skyrmion}.
Apart from the discrepancy in the threshold diameter between incomplete and isolated skyrmions,
the results delivered by the two models are in good agreement with each other.

\begin{table}[ht]
\begin{center}
\begin{tabular}{|c||c|c|c|c|c|c|c|c|}
\hline
$d$ [\si{\nano\meter}] &
\num{80} &
\num{90} &
\num{100} &
\num{120} &
\num{140} &
\num{160} &
\num{180} &
\num{200} \\
\hline
\hline
$h_{3D}$ [\si{\nano\meter}] &
\num{5.27} &
\num{5.23} &
\num{5.23} &
\num{5.20} &
\num{5.25} &
\num{5.29} &
\num{5.33} &
\num{5.38} \\
\hline
$h_{2D}$ [\si{\nano\meter}] &
\num{4.41} &
\num{4.45} &
\num{4.45} &
\num{4.46} &
\num{4.42} &
\num{4.41} &
\num{4.41} &
\num{4.45} \\
\hline
$\#\mathcal{T}_{h,3D}$ &
\num{31084} &
\num{39895} &
\num{51340} &
\num{74491} &
\num{95166} &
\num{134150} &
\num{164068} &
\num{214528} \\
\hline
$\#\mathcal{T}_{h,2D}$ &
\num{1354} &
\num{1830} &
\num{2142} &
\num{3706} &
\num{5062} &
\num{7062} &
\num{7924} &
\num{10120} \\
\hline
error &
\num{0.0383} &
\num{0.0368} &
\num{0.2200} &
\num{0.0854} &
\num{0.0792} &
\num{0.0992} &
\num{0.0953} &
\num{0.0906} \\
\hline
\end{tabular}
\end{center}
\caption{Experiment of Section~\ref{sec:exp1}:
Mesh size and number of elements of the meshes,
and relative energy error~\eqref{eq:energy_error}
between the 3D and 2D computations.}
\label{tab:exp1_results} 
\end{table}

In Table~\ref{fig:experiment1_stable},
for all considered values of the disk diameter,
we collect
the mesh sizes of the 3D and 2D meshes,
the corresponding number of elements,
and the relative error
\begin{equation} \label{eq:energy_error}
\mathrm{error} = \frac{\lvert \mathcal{E}(\m_{3D}) - \mathcal{E}(\m_{2D}) \rvert}{\lvert \mathcal{E}(\m_{3D}) \rvert}.
\end{equation}
Here, $\m_{3D}$ denotes the equilibrium state computed with the full 3D model,
while $\m_{2D}$ denotes the 3D extension (homogeneous in $x_3$) of the equilibrium state
computed with the reduced 2D model.
For each value of $d$, comparing the number of elements of the 3D and 2D meshes,
we have a quantitative measurement of the reduction of the computational complexity
guaranteed by the reduced model.
Note that, besides the significantly smaller number of degrees of freedom,
a strong advantage of the reduced 2D model is that all energy contributions are local,
while the full 3D model 
requires the solution
of a nonlocal problem in each time-step
to compute the magnetostatic field.
This aspect widens the gap between the computational complexities of the two approaches further.
Looking at the values of the relative error in the approximation of the energy,
we see that the error,
apart from the case $d =$ \SI{100}{\nano\meter} (where the stable states obtained by the two models are different),
always stays below 10\%.

\subsection{Comparison of the models in the evolutionary case} \label{sec:exp2}

In our second experiment,
we address the effectivity of the reduced 2D model for the evolutionary problem.
Moreover,
we try to identify the source of the quantitative discrepancy between the models
observed in the first experiment.
We restrict ourselves to the disk of diameter \SI{140}{\nano\meter}
and repeat the experiment of Section~\ref{sec:exp1}.
However, to have a slower and more realistic magnetization dynamics,
we now use the experimental value $\alpha=$ \num{0.28}
for the Gilbert damping parameter~\cite{babcwcvhcsmf2017}.
In this case, the equilibrium magnetization configuration is an isolated skyrmion
(see Figure~\ref{fig:exp2_pictures}).

\begin{figure}[ht]
\includegraphics[height=4cm]{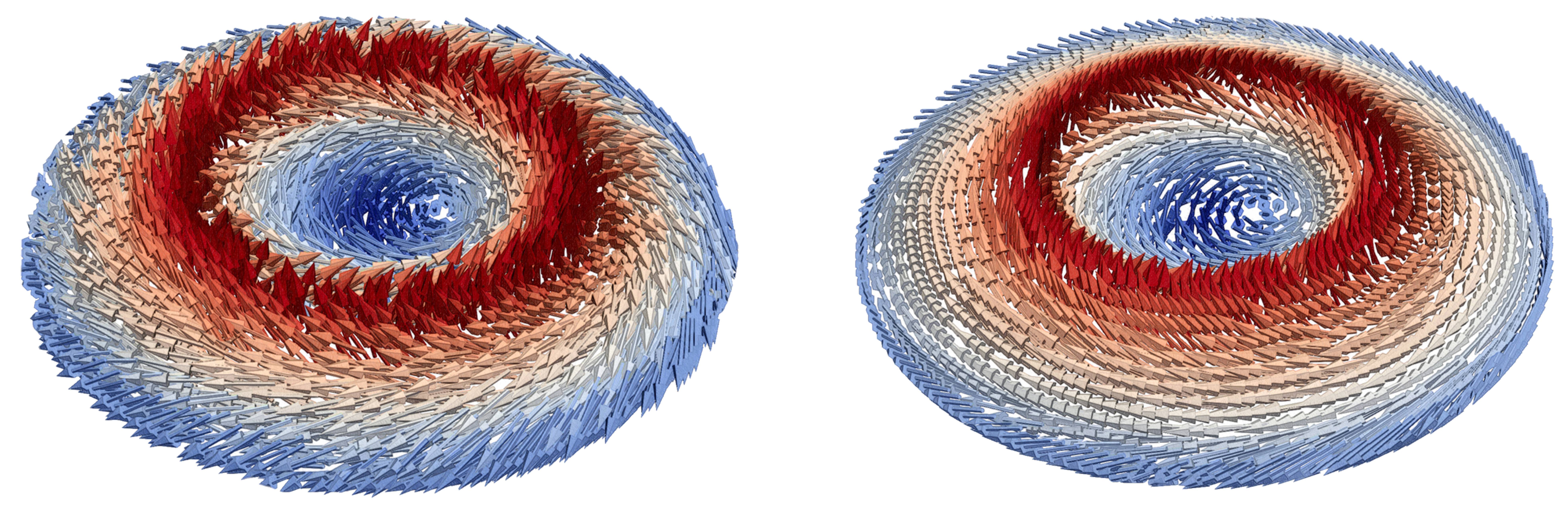}
\caption{Experiment of Section~\ref{sec:exp2}:
3D view of the isolated skyrmion
computed with the full 3D model (left) and the reduced 2D model (right).
}
\label{fig:exp2_pictures}
\end{figure}

\begin{figure}[t]
\centering
\includegraphics[width=10cm]{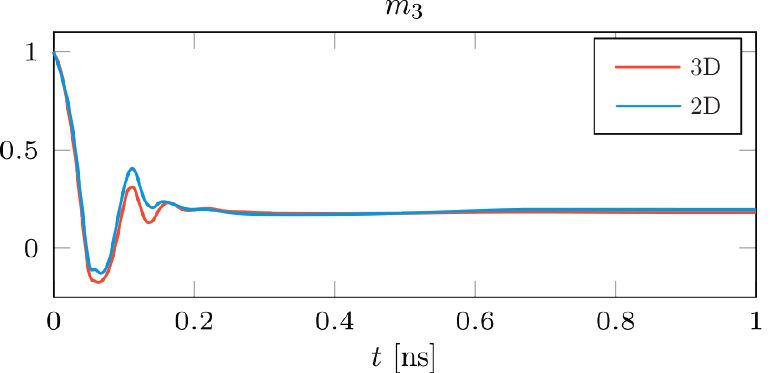}
\caption{Experiment of Section~\ref{sec:exp2}:
Time evolution of $\langle m_3 \rangle$.
Comparison of the results obtained with the full 3D model and the reduced 2D model.}
\label{fig:exp2_average}
\end{figure}
\begin{figure}[t]
\centering
\centering
\includegraphics[width=16.4cm]{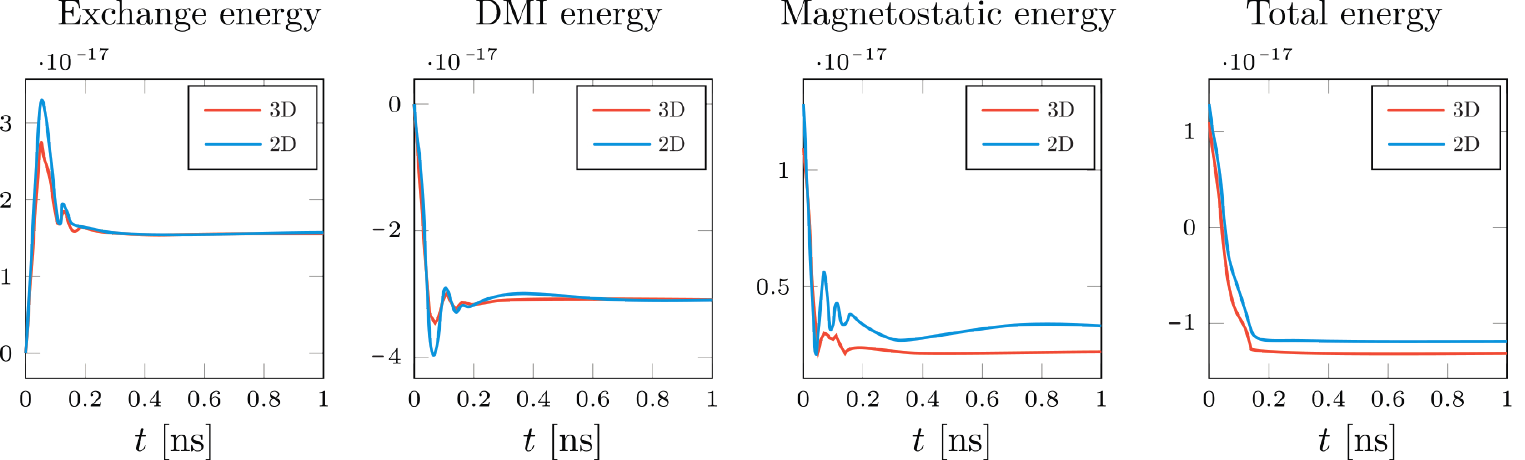}
\caption{Experiment of Section~\ref{sec:exp2}:
Time evolution of the energy contributions (exchange, DMI, magnetostatic)
and the total energy (all measured in \si{\joule}).
Comparison of the results obtained with the full 3D model and the reduced 2D model.}
\label{fig:exp2_energies}
\end{figure}

In Figure~\ref{fig:exp2_average}
and in Figure~\ref{fig:exp2_energies},
for both the full 3D model and the reduced 2D model,
we plot the time evolution of the spatial average of
the out-of-plane magnetization component $m_3$,
i.e.,
$\langle m_3 \rangle = \lvert\Omega\rvert^{-1} \int_{\Omega} m_3$,
of the total energy,
and of each energy contribution separately.
We observe a clear quantitative agreement between the two models,
which highlights the effectivity of the proposed approach also for the dynamic problem.
Looking at the evolution of the different energy contributions,
we see that
the discrepancy in the total energy between the models is mainly attributable
to the magnetostatic contribution.
Finally, the plot of the time evolution of the total energy~\eqref{eq:energy_physical},
which decays monotonically,
provides a numerical validation of the dissipative energy law
guaranteed by Algorithm~\ref{alg:tps}
(cf.\ Proposition~\ref{prop:numerics}).

The aim of the experiments included in the present paper
was to
numerically validate the thin-film limits
established in Sections~\ref{subsec:equicoercivityffprime}--\ref{sec:LLG}
and to
show that the reduced 2D model, to some extent, provides
a computationally attractive tool to qualitatively study the physics 
of magnetic thin films for materials with bulk DMI.
Future numerical studies,
out of the scope of the present work,
will investigate the difference between the models more quantitatively.

\section*{Acknowledgments}
The authors thank Carl-Martin Pfeiler (TU Wien) for his help with the Commics code~\cite{prsehhsmp2020}.
All authors acknowledge the support of the Austrian Science Fund (FWF) through the special research program
\emph{Taming complexity in partial differential systems} (grant F65).
All authors also acknowledge support from
the Erwin Schr\"odinger International Institute for Mathematics and Physics (ESI) in Vienna,
given on the occasion of the workshop on \emph{New Trends in the Variational Modeling and Simulation of Liquid Crystals} held at ESI on December 2--6, 2019.
The research of ED has been supported by the FWF through the grants V 662-N32
\emph{High contrast materials in plasticity and magnetoelasticity}
and I 4052 N32 \emph{Large Strain Challenges in Materials Science},
and from BMBWF through the OeAD-WTZ project CZ04/2019
\emph{Mathematical Frontiers in Large Strain Continuum Mechanics}.
GDF and MR would like to thank the Isaac Newton Institute for Mathematical Sciences for support and hospitality
during the program \emph{The mathematical design of new materials}
(supported by EPSRC grant number EP/R014604/1),
when work on this paper was undertaken.

\bibliographystyle{siam} 
\bibliography{literature}

\begin{thebibliography}{10}

\bibitem{aesds2013}
{\sc C.~Abert, L.~Exl, G.~Selke, A.~Drews, and T.~Schrefl}, {\em Numerical
  methods for the stray-field calculation: {A} comparison of recently developed
  algorithms}, Journal of Magnetism and Magnetic Materials, 326 (2013),
  pp.~176--185.

\bibitem{abertetal2014}
{\sc C.~Abert, G.~Hrkac, M.~Page, D.~Praetorius, M.~Ruggeri, and D.~Suess},
  {\em Spin-polarized transport in ferromagnetic multilayers: an
  unconditionally convergent {FEM} integrator}, Computers \& Mathematics with
  Applications. An International Journal, 68 (2014), pp.~639--654.

\bibitem{AcerbiA2006}
{\sc E.~Acerbi, I.~Fonseca, and G.~Mingione}, {\em {Existence and regularity
  for mixtures of micromagnetic materials}}, Proceedings of the Royal Society
  A: Mathematical, Physical and Engineering Sciences, 462 (2006),
  pp.~2225--2243.

\bibitem{alouges2008}
{\sc F.~Alouges}, {\em A new finite element scheme for {L}andau--{L}ifshitz
  equations}, Discrete and Continuous Dynamical Systems. Series S, 1 (2008),
  pp.~187--196.

\bibitem{alouges2015homogenization}
{\sc F.~Alouges and G.~{Di Fratta}}, {\em {Homogenization of composite
  ferromagnetic materials}}, Proceedings of the Royal Society A: Mathematical,
  Physical and Engineering Science, 471 (2015), p.~20150365.

\bibitem{akst2014}
{\sc F.~Alouges, E.~Kritsikis, J.~Steiner, and J.-C. Toussaint}, {\em A
  convergent and precise finite element scheme for
  {L}andau-{L}ifschitz-{G}ilbert equation}, Numerische Mathematik, 128 (2014),
  pp.~407--430.

\bibitem{alouges1992global}
{\sc F.~Alouges and A.~Soyeur}, {\em On global weak solutions for
  {L}andau--{L}ifshitz equations: {E}xistence and nonuniqueness}, Nonlinear
  Analysis, 18 (1992), pp.~1071--1084.

\bibitem{Back_2020}
{\sc C.~Back, V.~Cros, H.~Ebert, K.~Everschor-Sitte, A.~Fert, M.~Garst, T.~Ma,
  S.~Mankovsky, T.~L. Monchesky, M.~Mostovoy, N.~Nagaosa, S.~S.~P. Parkin,
  C.~Pfleiderer, N.~Reyren, A.~Rosch, Y.~Taguchi, Y.~Tokura, K.~von Bergmann,
  and J.~Zang}, {\em The 2020 skyrmionics roadmap}, Journal of Physics D:
  Applied Physics, 53 (2020), p.~363001.

\bibitem{bartels2015}
{\sc S.~Bartels}, {\em Numerical methods for nonlinear partial differential
  equations}, vol.~47, Springer, 2015.

\bibitem{bp2006}
{\sc S.~Bartels and A.~Prohl}, {\em Convergence of an implicit finite element
  method for the {L}andau-{L}ifshitz-{G}ilbert equation}, SIAM Journal on
  Numerical Analysis, 44 (2006), pp.~1405--1419.

\bibitem{babcwcvhcsmf2017}
{\sc M.~Beg, M.~Albert, M.-A. Bisotti, D.~Cort{\'e}s-Ortu{\~n}o, W.~Wang,
  R.~Carey, M.~Vousden, O.~Hovorka, C.~Ciccarelli, C.~S. Spencer, C.~H.
  Marrows, and H.~Fangohr}, {\em Dynamics of skyrmionic states in confined
  helimagnetic nanostructures}, Physical Review B, 95 (2017), p.~014433.

\bibitem{bcwcvbachsf2015}
{\sc M.~Beg, R.~Carey, W.~Wang, D.~Cort{\'e}s-Ortu{\~n}o, M.~Vousden, M.-A.
  Bisotti, M.~Albert, D.~Chernyshenko, O.~Hovorka, R.~L. Stamps, and
  H.~Fangohr}, {\em Ground state search, hysteretic behaviour, and reversal
  mechanism of skyrmionic textures in confined helimagnetic nanostructures},
  Scientific Reports, 5 (2015), p.~17137.

\bibitem{braides1998homogenization}
{\sc A.~Braides and A.~Defranceschi}, {\em Homogenization of {M}ultiple
  {I}ntegrals}, Oxford University Press, Oxford, 1998.

\bibitem{bresciani}
{\sc M.~Bresciani}, {\em Linearized von {K}\'arm\'an theory for incompressible
  magnetoelastic plates}, Preprint arXiv:2007.14122,  (2020).

\bibitem{BrownB1962}
{\sc W.~F. Brown}, {\em {Magnetostatic principles in ferromagnetism}},
  North-Holland Publishing Company, New York, 1962.

\bibitem{BrownB1963}
{\sc W.~F. Brown}, {\em {Micromagnetics}}, Interscience Publishers, London,
  1963.

\bibitem{multiscale2014}
{\sc F.~Bruckner, D.~Suess, M.~Feischl, T.~F\"{u}hrer, P.~Goldenits, M.~Page,
  D.~Praetorius, and M.~Ruggeri}, {\em Multiscale modeling in micromagnetics:
  existence of solutions and numerical integration}, Mathematical Models and
  Methods in Applied Sciences, 24 (2014), pp.~2627--2662.

\bibitem{CanteroOtto06}
{\sc R.~Cantero-\'{A}lvarez and F.~Otto}, {\em Critical fields in ferromagnetic
  thin films: identification of four regimes}, Journal of Nonlinear Science, 16
  (2006), pp.~351--383.

\bibitem{CanteroOttoSteiner07}
{\sc R.~Cantero-\'{A}lvarez, F.~Otto, and J.~Steiner}, {\em The concertina
  pattern: a bifurcation in ferromagnetic thin films}, Journal of Nonlinear
  Science, 17 (2007), pp.~221--281.

\bibitem{CapellaMelcherOtto07}
{\sc A.~Capella, C.~Melcher, and F.~Otto}, {\em Wave-type dynamics in
  ferromagnetic thin films and the motion of {N}\'{e}el walls}, Nonlinearity,
  20 (2007), pp.~2519--2537.

\bibitem{carbou2001thin}
{\sc G.~Carbou}, {\em {Thin layers in micromagnetism}}, Mathematical Models and
  Methods in Applied Sciences, 11 (2001), pp.~1529--1546.

\bibitem{CoteIgnatMiot14}
{\sc R.~C\^{o}te, R.~Ignat, and E.~Miot}, {\em A thin-film limit in the
  {L}andau-{L}ifshitz-{G}ilbert equation relevant for the formation of
  {N}\'{e}el walls}, Journal of Fixed Point Theory and Applications, 15 (2014),
  pp.~241--272.

\bibitem{dal1993introduction}
{\sc G.~{Dal Maso}}, {\em {Introduction to $\Gamma$-convergence}}, Birkhäuser
  Basel, 1993.

\bibitem{Davoli_2020}
{\sc E.~Davoli and G.~{Di Fratta}}, {\em Homogenization of chiral magnetic
  materials: A mathematical evidence of {D}zyaloshinskii's predictions on
  helical structures}, Journal of Nonlinear Science, 30 (2020), pp.~1229--1262.

\bibitem{dkps}
{\sc E.~Davoli, M.~Kruzík, P.~Piovano, and U.~Stefanelli}, {\em Magnetoelastic
  thin films at large strains}, Continuum Mechanics and Thermodynamics,
  (2020).

\bibitem{DeSimoneKohnMuellerOtto02}
{\sc A.~DeSimone, R.~V. Kohn, S.~M\"{u}ller, and F.~Otto}, {\em A reduced
  theory for thin-film micromagnetics}, Communications on Pure and Applied
  Mathematics, 55 (2002), pp.~1408--1460.

\bibitem{DeSimoneETal01}
{\sc A.~DeSimone, R.~V. Kohn, S.~M\"{u}ller, F.~Otto, and R.~Sch\"{a}fer}, {\em
  Two-dimensional modelling of soft ferromagnetic films}, The Royal Society of
  London. Proceedings. Series A. Mathematical, Physical and Engineering
  Sciences, 457 (2001), pp.~2983--2991.

\bibitem{Di_Fratta_2020}
{\sc G.~{Di~Fratta}}, {\em Micromagnetics of curved thin films}, Zeitschrift
  für angewandte Mathematik und Physik, 71 (2020).

\bibitem{Di_Fratta_2019b}
{\sc G.~{Di Fratta}, M.~Innerberger, and D.~Praetorius}, {\em
  Weak{\textendash}strong uniqueness for the
  {L}andau{\textendash}{L}ifshitz{\textendash}{G}ilbert equation in
  micromagnetics}, Nonlinear Analysis: Real World Applications, 55 (2020),
  p.~103122.

\bibitem{Di_Fratta_2019}
{\sc G.~{Di Fratta}, C.~B. Muratov, F.~N. Rybakov, and V.~V. Slastikov}, {\em
  Variational principles of micromagnetics revisited}, SIAM Journal on
  Mathematical Analysis, 52 (2020), pp.~3580--3599.

\bibitem{gao2014}
{\sc H.~Gao}, {\em Optimal error estimates of a linearized backward {E}uler
  {FEM} for the {L}andau-{L}ifshitz equation}, SIAM Journal on Numerical
  Analysis, 52 (2014), pp.~2574--2593.

\bibitem{carlos2007}
{\sc C.~J. Garc\'{\i}a-Cervera}, {\em Numerical micromagnetics: a review},
  Bolet\'{\i}n de la Sociedad Espa\~{n}ola de Matem\'{a}tica Aplicada. SeMA,
  (2007), pp.~103--135.

\bibitem{EGarcaCervera2001EffectiveDF}
{\sc C.~J. Garc{\'i}a-Cervera and W.~E}, {\em Effective dynamics for
  ferromagnetic thin films}, Journal of Applied Physics, 90 (2001),
  pp.~370--374.

\bibitem{gilbert2004phenomenological}
{\sc T.~L. Gilbert}, {\em A phenomenological theory of damping in ferromagnetic
  materials}, {IEEE} Transactions on Magnetics, 40 (2004), pp.~3443--3449.

\bibitem{GioiaJames97}
{\sc G.~Gioia and R.~D. James}, {\em Micromagnetics of very thin films},
  Proceedings of the Royal Society of London. Series A: Mathematical, Physical
  and Engineering Sciences, 453 (1997), pp.~213--223.

\bibitem{HadijiShirakawa10}
{\sc R.~Hadiji and K.~Shirakawa}, {\em 3{D}-2{D} asymptotic observation for
  minimization problems associated with degenerate energy-coefficients},
  Discrete and Continuous Dynamical Systems. Series A,  (2011), pp.~624--633.

\bibitem{hrkac2017convergent}
{\sc G.~Hrkac, C.-M. Pfeiler, D.~Praetorius, M.~Ruggeri, A.~Segatti, and
  B.~Stiftner}, {\em Convergent tangent plane integrators for the simulation of
  chiral magnetic skyrmion dynamics}, Advances in Computational Mathematics, 45
  (2019), pp.~1329--1368.

\bibitem{hubert2008magnetic}
{\sc A.~Hubert and R.~Sch{\"a}fer}, {\em Magnetic domains: the analysis of
  magnetic microstructures}, Springer Science \& Business Media, 2008.

\bibitem{Ignat09}
{\sc R.~Ignat}, {\em A survey of some new results in ferromagnetic thin films},
  in S\'{e}minaire: \'{E}quations aux {D}\'{e}riv\'{e}es {P}artielles,
  \'{E}cole Polytechnique, Palaiseau, 2009, pp.~Exp. No. VI, 21.

\bibitem{IgnatOtto07}
{\sc R.~Ignat and F.~Otto}, {\em A compactness result in thin-film
  micromagnetics and the optimality of the {N}\'{e}el wall}, Journal of the
  European Mathematical Society (JEMS), 10 (2008), pp.~909--956.

\bibitem{IgnatOtto11}
\leavevmode\vrule height 2pt depth -1.6pt width 23pt, {\em A compactness result
  for {L}andau state in thin-film micromagnetics}, Annales de l'Institut Henri
  Poincar\'{e}. Analyse Non Lin\'{e}aire, 28 (2011), pp.~247--282.

\bibitem{KnuepferMuratovNolte19}
{\sc H.~Kn\"{u}pfer, C.~B. Muratov, and F.~Nolte}, {\em Magnetic domains in
  thin ferromagnetic films with strong perpendicular anisotropy}, Archive for
  Rational Mechanics and Analysis, 232 (2019), pp.~727--761.

\bibitem{KohnSlastikov05}
{\sc R.~V. Kohn and V.~V. Slastikov}, {\em Another thin-film limit of
  micromagnetics}, Archive for Rational Mechanics and Analysis, 178 (2005),
  pp.~227--245.

\bibitem{KohnSlastikov-dyn05}
\leavevmode\vrule height 2pt depth -1.6pt width 23pt, {\em Effective dynamics
  for ferromagnetic thin films: a rigorous justification}, Proceedings of The
  Royal Society of London. Series A. Mathematical, Physical and Engineering
  Sciences, 461 (2005), pp.~143--154.

\bibitem{KruzikProhl06}
{\sc M.~Kruzík and A.~Prohl}, {\em Recent developments in the modeling,
  analysis, and numerics of ferromagnetism}, SIAM Review, 48 (2006),
  pp.~439--483.

\bibitem{Kurzke06}
{\sc M.~Kurzke}, {\em Boundary vortices in thin magnetic films}, Calculus of
  Variations and Partial Differential Equations, 26 (2006), pp.~1--28.

\bibitem{KurzkeMelcherMoser06}
{\sc M.~Kurzke, C.~Melcher, and R.~Moser}, {\em Domain walls and vortices in
  thin ferromagnetic films}, in Analysis, modeling and simulation of multiscale
  problems, Springer, Berlin, 2006, pp.~249--298.

\bibitem{KurzkeMelcherMoserSpirn11}
{\sc M.~Kurzke, C.~Melcher, R.~Moser, and D.~Spirn}, {\em Ginzburg-{L}andau
  vortices driven by the {L}andau-{L}ifshitz-{G}ilbert equation}, Archive for
  Rational Mechanics and Analysis, 199 (2011), pp.~843--888.

\bibitem{LandauA1935}
{\sc L.~Landau and E.~Lifshitz}, {\em On the theory of the dispersion of
  magnetic permeability in ferromagnetic bodies}, in Perspectives in
  Theoretical Physics, Elsevier, 1992, pp.~51--65.

\bibitem{LundMuratov16}
{\sc R.~G. Lund and C.~B. Muratov}, {\em One-dimensional domain walls in thin
  ferromagnetic films with fourfold anisotropy}, Nonlinearity, 29 (2016),
  pp.~1716--1734.

\bibitem{LundMuratovSlastikov20}
{\sc R.~G. Lund, C.~B. Muratov, and V.~V. Slastikov}, {\em Edge domain walls in
  ultrathin exchange-biased films}, Journal of Nonlinear Science, 30 (2020),
  pp.~1165--1205.

\bibitem{Melcher-reg07}
{\sc C.~Melcher}, {\em A dual approach to regularity in thin film
  micromagnetics}, Calculus of Variations and Partial Differential Equations,
  29 (2007), pp.~85--98.

\bibitem{Melcher10}
\leavevmode\vrule height 2pt depth -1.6pt width 23pt, {\em Thin-film limits for
  {L}andau-{L}ifshitz-{G}ilbert equations}, SIAM Journal on Mathematical
  Analysis, 42 (2010), pp.~519--537.

\bibitem{Melcher14}
\leavevmode\vrule height 2pt depth -1.6pt width 23pt, {\em Chiral skyrmions in
  the plane}, Proceedings of The Royal Society of London. Series A.
  Mathematical, Physical and Engineering Sciences, 470 (2014), p.~20140394.

\bibitem{Melcher_2019}
{\sc C.~Melcher and Z.~N. Sakellaris}, {\em Curvature-stabilized skyrmions with
  angular momentum}, Letters in Mathematical Physics, 109 (2019),
  pp.~2291--2304.

\bibitem{MoriniSlastikov18}
{\sc M.~Morini and V.~Slastikov}, {\em Reduced models for ferromagnetic thin
  films with periodic surface roughness}, Journal of Nonlinear Science, 28
  (2018), pp.~513--542.

\bibitem{Moser04}
{\sc R.~Moser}, {\em Boundary vortices for thin ferromagnetic films}, Archive
  for Rational Mechanics and Analysis, 174 (2004), pp.~267--300.

\bibitem{Moser05}
\leavevmode\vrule height 2pt depth -1.6pt width 23pt, {\em Moving boundary
  vortices for a thin-film limit in micromagnetics}, Communications on Pure and
  Applied Mathematics, 58 (2005), pp.~701--721.

\bibitem{Muratov19}
{\sc C.~B. Muratov}, {\em A universal thin film model for {G}inzburg-{L}andau
  energy with dipolar interaction}, Calculus of Variations and Partial
  Differential Equations, 58 (2019), p.~Paper No. 52.

\bibitem{prsehhsmp2020}
{\sc C.-M. Pfeiler, M.~Ruggeri, B.~Stiftner, L.~Exl, M.~Hochsteger, G.~Hrkac,
  J.~Sch\"oberl, N.~J. Mauser, and D.~Praetorius}, {\em Computational
  micromagnetics with {C}ommics}, Computer Physics Communications, 248 (2020),
  p.~106965.

\bibitem{praetorius2004analysis}
{\sc D.~Praetorius}, {\em Analysis of the operator {$\Delta^{-1}\textrm{div}$}
  arising in magnetic models}, Zeitschrift für Analysis und ihre Anwendungen,
  23 (2004), pp.~589--605.

\bibitem{prohl2001}
{\sc A.~Prohl}, {\em Computational micromagnetism}, Advances in Numerical
  Mathematics, B. G. Teubner, Stuttgart, 2001.

\bibitem{schroers1995bogomol}
{\sc B.~J. Schroers}, {\em Bogomol'nyi solitons in a gauged {$O(3)$} sigma
  model}, Physics Letters B, 356 (1995), pp.~291--296.

\bibitem{Slastikov05}
{\sc V.~Slastikov}, {\em Micromagnetics of thin shells}, Mathematical Models
  and Methods in Applied Sciences, 15 (2005), pp.~1469--1487.

\bibitem{Streubel_2016}
{\sc R.~Streubel, P.~Fischer, F.~Kronast, V.~P. Kravchuk, D.~D. Sheka,
  Y.~Gaididei, O.~G. Schmidt, and D.~Makarov}, {\em Magnetism in curved
  geometries}, Journal of Physics D: Applied Physics, 49 (2016), p.~363001.

\bibitem{gaussseidel2020}
{\sc C.~Xie, C.~J. Garc\'{\i}a-Cervera, C.~Wang, Z.~Zhou, and J.~Chen}, {\em
  Second-order semi-implicit projection methods for micromagnetics
  simulations}, Journal of Computational Physics, 404 (2020), p.~109104.

\end{thebibliography}
\end{document}